\documentclass[a4paper,11pt]{article}
\usepackage{fullpage}
\usepackage{amsmath, amsthm, amsfonts, amssymb}
\usepackage{bm}
\usepackage{xcolor}
\usepackage[linesnumbered,ruled, algosection]{algorithm2e}
\usepackage{setspace}
\usepackage{hyperref}

\newtheorem{theorem}{Theorem}[section]
\newtheorem{proposition}[theorem]{Proposition}
\newtheorem{lemma}[theorem]{Lemma}
\newtheorem{remark}[theorem]{Remark}
\newtheorem{corollary}[theorem]{Corollary}

\DeclareMathOperator{\Tr}{Tr}
\DeclareMathOperator{\Span}{Span}
\DeclareMathOperator*{\argmin}{arg\,min}
\DeclareMathOperator*{\argmax}{arg\,max}

\renewcommand{\(}{\left(}
\renewcommand{\)}{\right)}
\newcommand{\<}{\langle}
\renewcommand{\>}{\rangle}
\renewcommand{\leq}{\leqslant}
\renewcommand{\geq}{\geqslant}
\newcommand{\be}{\begin{equation}}
\newcommand{\ee}{\end{equation}}

\newcommand{\eps}{\varepsilon}
\def\vp{\varphi}
\def\t{\tilde}

\newcommand{\C}{\mathbb{C}}
\newcommand{\E}{\mathbb{E}}
\newcommand{\N}{\mathbb{N}}
\renewcommand{\P}{\mathbb{P}}
\newcommand{\R}{\mathbb{R}}

\def\cO{{\cal O}}

\newcommand{\bk}{{\boldsymbol k}}

\newcommand{\bx}{{\boldsymbol x}}
\newcommand{\by}{{\boldsymbol y}}

\newcommand{\bA}{{\boldsymbol A}}

\newcommand{\bI}{{\boldsymbol I}}
\newcommand{\bL}{{\boldsymbol L}}
\newcommand{\bM}{{\boldsymbol M}}

\newcommand{\bY}{{\boldsymbol Y}}
\newcommand{\bW}{{\boldsymbol W}}
\newcommand{\bZ}{{\boldsymbol Z}}

\def\cnew{ \color{blue} }

\usepackage{graphicx,subcaption} 
\usepackage{algorithmic}
\title{Randomized least-squares with minimal oversampling and interpolation in general spaces
}
\newcommand{\email}[1]{\href{mailto:#1}{#1}}
\author{Matthieu Dolbeault%
\thanks{RWTH Aachen, Germany (\email{dolbeault@igpm.rwth-aachen.de}). The author acknowledges funding by the Deutsche Forschungsgemeinschaft (DFG, German Research Foundation) Project number 442047500 through the Collaborative Research Center “Sparsity and Singular Structures” (SFB 1481).}
\and 
Moulay Abdellah Chkifa%
\thanks{Mohammed VI Polytechnic University, Ben Guerir (\email{abdellah.chkifa@um6p.ma}).}
}
\date{}
\DeclareMathOperator{\supp}{supp}
\def\cnew{\color{black}}

\begin{document}

\maketitle

\begin{abstract}
In approximation of functions based on point values, least-squares methods provide more stability than interpolation,
at the expense of increasing the sampling budget. We show that near-optimal approximation error can nevertheless be achieved, in an expected $L^2$ sense, as soon as the sample size $m$ is larger than the dimension $n$ of the 
approximation space by a constant ratio. On the other hand, for $m=n$, we obtain an interpolation strategy with a 
stability factor of order $n$. The proposed sampling algorithms are greedy procedures based on 
\cite{BSS} and \cite{LS_polynomial}, with polynomial computational complexity.
\end{abstract}



\section{Introduction and main results}
\label{intro}

Let $(X,\mathcal A, \mu)$ be a probability space.
We consider the problem of estimating an unknown function $f: X \to \C$
from evaluations of $f$ at chosen points $x_1,\dots,x_m\in X$.
We assess the error between $f$ and its estimator $\t f$ either in the $L^2(X,\mu)$ norm
\[
\|g\|_{L^2}:=\left(\int_X |g(x)|^2 d\mu(x)\right)^{1/2},
\]
or in the uniform norm $\|g\|_{L^\infty(X,\mu)}$.
Given a subspace $V_n\subset L^2(X,\mu)$ of dimension $n$, we 
would like the estimator $\t f$ to belong to $V_n$ and to perform almost as well as
the best approximation of $f$ in $V_n$, that is, its orthogonal projection
\[
P_nf=\argmin_{v\in V_n}\|f-v\|_{L^2}.
\]

As we only have access to point-wise observations,
we cannot explicitly compute $P_n f$ in general.
In this context, a classical
approach consists in considering
a solution to the weighted least-squares problem
\[
P_n^mf \in \argmin_{v\in V_n} \sum_{i=1}^m s_i |f(x_i)-v(x_i)|^2,
\]
where we may use some weights $s_1,\dots,s_m>0$.
We would like this problem to admit a unique solution, and therefore require $m\geq n$.
Similar to $P_n$ for the $L^2$ norm,
the operator $P_n^m$ is the orthogonal projector
onto $V_n$ with respect to the empirical norm
\[
\|g\|_m:=\( \sum_{i=1}^m s_i |g(x_i)|^2\)^{1/2}.
\]

The approximation accuracy is inherently related to the points $x_i$ and 
the weights $s_i$. For the sake of illustration, consider the 
setting where $V_n =\P_{n-1}$ is the space of algebraic polynomials of 
degree less than $n$, restricted to the interval $X=[-1,1]$, and 
choose $m=n$, so that $P_n^n$ is the Lagrange interpolation operator
associated with $\{x_1,\dots,x_n\}$.
For equally spaced points $x_i$, this corresponds to interpolation on a uniform grid, which is known to be highly unstable, failing to converge towards $f$, even when $f$ is infinitely smooth.
This is the so-called \emph{Runge phenomenon}, see e.g. \cite{MM2008}.
The phenomenon persists with equally spaced points even for $m  = r n$ 
with $r >1 $ constant, as {\cnew observed in \cite{BX} and theoretically explained in \cite{PTK}. In fact,
the Runge phenomenon occurs for any points that do not cluster quadratically like Chebyshev points, see \cite{APS}.} It is 
however defeated by interpolation on Chebyshev type points.

As far as interpolation in spaces $V_n$ is concerned, 
there exists no systematic choice of points that prevents all instabilities. 
A generic choice is that of Fekete points, which 
ensure that $ \| P_n^n g \|_{\mathcal L^\infty} \leq n \| g \|_{\mathcal L^\infty}$ 
in the strong uniform norm $\|g\|_{\mathcal L^\infty}=\sup_{x\in X}|g(x)|$%
\footnote{We assume here that $V_n$ is included in $\mathcal L^\infty$,
the space of bounded functions.}, 
resulting in the stability inequality
\[
\|f-P_n^n f\|_{\mathcal L^\infty} \leq (n+1) \inf_{v\in V_n}\|f- v\|_{\mathcal L^\infty},
\]
see for instance Proposition 1.2.5 in \cite{N}.
This guarantees a good approximation $P_n^n f$ of $f$, 
provided the latter is sufficiently smooth and $V_n$ is well chosen. 
However, for general domains $X$ and spaces $V_n$, the computation of 
Fekete points can be intractable.

For polynomial interpolation over compact domains $X \subset \R$ or $\C$, 
Leja sequences are greedy alternatives to Fekete points.
For instance, for interpolation
by polynomials in $ \P_{n-1}$, restricted to union of closed intervals of 
$\R$, they provably yield $ \| P_n^n g \|_{\mathcal L^\infty} \lesssim n^{13/4} 
\| g \|_{\mathcal L^\infty}$, see the recent paper \cite{AN}.
We note that numerical evidence shows that $n^{13/4}$ can be replaced by $n$ for 
intervals. Similar polynomial growths, {\cnew with a factor $n^3$ in the estimate,} hold for a certain type of polynomial 
interpolation over tensor product domains, see \cite{CCS, CC_hierarchical}.   

On the other hand, near optimal approximation error in expected $L^2$ sense
\[
\E(\|f-P_n^m f\|_{L^2}^2) \leq C \inf_{v\in V_n} \|f- v\|_{L^2}^2,
\]
can be attained by taking $m$ larger than $n$. In the case of uniformly 
distributed points and equal weights $s_i=1$, this usually requires $m$ to scale 
polynomially in $n$, see \cite{CDL,CCMNT}, but a logarithmic 
oversampling is achievable if one considers a different sample 
distribution \cite{CM}.
{\cnew Numerical methods \cite{HNP} and a theoretical solution \cite{CD_reduced}} have been proposed to reduce the sample size $m$ to a constant multiple of~$n$.
{\cnew A discrete version of the above bound can also be found in \cite{CP}, in the context of statistical machine learning.}

In the last {\cnew five} papers, the sample points are drawn at random according to a prescribed measure, and the error bounds are presented in expectation. This setting does not require any additional assumption: indeed, point evaluations of a function $f\in L^2$ are defined in an almost sure sense, and so is the weighted least-squares projection $P_n^mf$. This will also be the main framework for the present article.

Our main theorem, stated below,
provides new bounds on the $L^2$ approximation error, depending on the ratio between $m$ and $n$.

\begin{theorem}
{\cnew Let $m \geq n$.}
\begin{itemize}
\item The conditional weighted least-squares approximation $\tilde f\in V_n$ defined in \eqref{conditional_WLS}, using $m$ evaluations of $f$ at points selected by 
Algorithm~\ref{algo_effective_resistance}, {\cnew with inputs 
$\eps = r^{-1/4}$ and $\gamma = r^{1/2} - r^{1/4}$}, satisfies
\begin{equation}
\E(\|f-\tilde f\|_{L^2}^2) \leq 
\(1+\frac{1}{r(1-r^{-1/4})^7}\) \min_{v\in V_n}\|f-v\|_{L^2}^2,
\label{main_equation_r_large}
\end{equation}
{\cnew where $r=\frac{m+1}{n}>1$.}
\item In turn, the weighted least-squares estimator $\tilde f=P_n^mf$, based on points selected by Algorithm~\ref{algo_fixed_increments},  {\cnew with inputs 
$\delta = r^{-1/2}$ and $\kappa \in [0,1]$}, simultaneously satisfies
\be
\E(\|f-\tilde f\|_{L^2}^2)\leq \left(1+\frac{1}{1-\kappa}\frac{1}{(1-1/\sqrt r)^{2}}\right)\min_{v\in V_n}\|f-v\|_{L^2}^2
\label{main_equation_r_small}
\ee
if $\kappa<1$ and
\be
\|f-\tilde f\|_{L^2}\leq \left(1+\frac{1}{\sqrt\kappa}\frac{1}{1-1/\sqrt r}\right)\min_{v\in V_n}\|f-v\|_{L^\infty} \quad a.s.
\label{main_equation_uniform}
\ee
{\cnew
if $\kappa>0$. Here $r=\frac{m}{n-1}>1$, and we assume that $V_n \subset L^\infty$ in the last bound.}

\end{itemize}
\label{main_theorem}
\end{theorem}

The first two bounds are of the same kind, the main difference being in the constant factor on the right-hand side, which is of optimal order $1+1/r+o(1/r)$ in \eqref{main_equation_r_large} as $r$ tends to infinity, but behaves better in \eqref{main_equation_r_small} when $r$ gets close to $1$. {\cnew The parameter $\kappa$ is a free input parameter selected by the user.}
The choice $\kappa=0$ gives the best constant in \eqref{main_equation_r_small}, while $\kappa=1$ is best suited for \eqref{main_equation_uniform}. One can also construct an estimator $\t f=P_n^mf$ satisfying both estimates at the same time, by running Algorithm~\ref{algo_fixed_increments} with an intermediate value of $\kappa$.

The uniform bound \eqref{main_equation_uniform} follows the approach developed in {\cnew\cite{LT,T,PU,DT, BSU}}, slightly improving the constants when compared to Theorem 1.1 in \cite{T} {\cnew and Theorem 6.3 in \cite{BSU}}, and linking it to the approach in expectation \eqref{main_equation_r_small} by the use of a common algorithm.
In a third approach, similar deterministic bounds can be proved by assuming more regularity on $f$ through a nested sequence of approximation spaces $(V_n)_{n\in \N}$, in a Hilbert space setting \cite{KU1, KUV, MU, NSU, BSU, GW} or in more general Banach spaces \cite{KU2, DKU, KPUU}.
\newline

Taking $m=n$ in the last two estimates, and observing that
$
{1}/{\sqrt r}=\sqrt{1-1/n}\leq1- {1}/{2n},
$
we immediately obtain the following result on interpolation in $L^2$.
\begin{corollary}
For $m=n$, the interpolation $\tilde f=P_n^nf\in V_n$ of $f$
at random points $x_1,\dots,x_n$ selected by
Algorithm \ref{algo_fixed_increments} 
{\cnew with input $\delta = \sqrt{1- 1/n}$}
achieves the accuracy bounds
\[
\E(\|f-\tilde f\|_{L^2}^2)\leq (4n^2+1)\min_{v\in V_n}\|f-v\|_{L^2}^2 \qquad \text{if }\;\kappa =0,
\]
and, assuming $V_n\subset L^\infty$,
\[
\|f-\tilde f\|_{L^2}\leq (2n+1)\min_{v\in V_n}\|f-v\|_{L^\infty}\quad a.s.\qquad \text{if }\;\kappa =1.
\]
\label{main_corollary}
\end{corollary}

\begin{remark}
In the case of uniform approximation, if {\cnew each function in $V_n$ is uniformly bounded on $X$} and not just essentially bounded, one can replace $L^\infty$ by the strong supremum norm $\mathcal L^\infty$, and remove the almost sure restriction, by considering deterministic samples satisfying the constraints of our algorithms. A similar observation can be found in Remark 3 of \cite{KPUU}.
\end{remark}

Finally, one can combine the second estimate in Corollary~\ref{main_corollary} with an inverse inequality between $L^2$ and $\mathcal L^\infty$ in $V_n$.
\begin{corollary}
\label{corollary_KW}
If the functions in $V_n$ are bounded and $X$ is compact, Algorithm \ref{algo_fixed_increments}, applied with 
the optimal measure {\cnew $\mu^*$} from \cite{KW}, provides $n$ points  for which the interpolation 
{\cnew ${\tilde f}^*=P_n^nf\in V_n$ of $f$} satisfies
\[
{\cnew \|f-{\tilde f}^* \|_{\mathcal L^\infty}\leq (1+2n\sqrt n)\min_{v\in V_n}\|f-v\|_{\mathcal L^\infty}.}
\]
\end{corollary}

Note that in the case $m=n$, the values of the weights $s_i$ have no importance, since the minimum in 
the definition of $P_n^m$ is zero. In fact, {\cnew if measure $\mu^*$ is known,} we exhibit a constructive set of points such that the Lebesgue stability constant 
\[
{\mathbb L}_n=\max_{f\in C^0(X)\setminus \{0\}}\frac{\|P_n^n f\|_{\mathcal L^\infty}}{\|f\|_{\mathcal L^\infty}}
\]
is at most $ 2n\sqrt n$. Although Fekete points 
achieve ${\mathbb L}_n \leq n$,
their computational complexity 
is exponential in $n$. In our approach, the main computational challenge is to find an optimal measure $\mu^*$ as defined in \cite{KW}, see also \cite{Bos}, which is {\cnew in general as difficult as looking for the Fekete points}.

However, using a different measure $\mu$, we still obtain a result similar to Corollary~\ref{corollary_KW} 
however with the factor $\sqrt n$ replaced by a larger power of $n$.
For example, if $X=[-1,1]^d$ 
and $V_n ={\rm span}\{x^\nu:\nu\in\Lambda\}$ is a space of 
multivariate polynomials indexed by a lower set  
$\Lambda\subset\N_0^d$, the uniform measure yields a 
factor $n$ instead of $\sqrt n$ and an estimate 
${\mathbb L}_n\leq 2n^2$, 
while the tensor product {\cnew arcsine} measure achieves a factor 
{\cnew $\min(2^{d/2} \sqrt {n}, n^{\log3/\log4})$ instead of $\sqrt n$,
resulting in the estimate 
${\mathbb L}_n\leq 2 n \min(2^{d/2} \sqrt {n}, n^{\log3/\log4})$.}

{\cnew Note that the Lebesgue constant does not only determine the convergence of the approximation, but also its robustness to numerical errors and noise in the measurements, see Section 4 in \cite{CM}, as well as \cite{APS,PTK}.}

The above discussion demonstrates that good interpolation points
can be constructed for multivariate polynomial approximation. 
The estimate $2 n^{1+\log3/\log4}$ on the 
Lebesgue constant is moderate. It outperforms the 
estimate $n^3$ which was established in \cite{CC_hierarchical} using
constructions based on $\Re$-Leja sequences. We 
note however that the new procedure is not hierarchical.

\begin{remark}
One can also use an inverse inequality between $L^2$ and $\mathcal L^\infty$ in the least-squares regime $m> n$.
More precisely, combining \eqref{main_equation_uniform} with \cite{KW} as in Corollary~\ref{corollary_KW}, we obtain
\[
\|f-P_n^mf\|_{\mathcal L^\infty}\leq \(1+\frac{\sqrt n}{1-1/\sqrt r}\)\min_{v\in V_n}\|f-v\|_{\mathcal L^\infty},
\]
with $r=m/(n-1)$. Taking the supremum for $f$ in a compact class of functions, and optimizing on both sides over $V_n$, this implies that the uniform sampling numbers are bounded by $\mathcal O(\sqrt n)$ times the uniform Kolmogorov $n$-widths, improving on the $\mathcal O(n)$ factor of Fekete points, if we allow for a constant oversampling $r>1$. This can already be seen by applying Corollary 5.3 in \cite{PU} together with the optimal density from \cite{KW}. The same result is obtained in Theorem 3 of \cite{KPUU}, where error bounds are also derived in any $L^p$ norm. We additionally refer to \cite{GW}, where implications in the field of Information Based Complexity are drawn, in a specific Hilbert space setting. The two very recent papers \cite{KPUU,GW} rely on the pioneering works \cite{KU1,KU2}, and on the infinite-dimensional adaptation \cite{DKU} of the result from \cite{MSS}, see also \cite{FS}.
\label{rk_KW_for_LS}
\end{remark}

The rest of the paper is organized as follows. In Section~\ref{SectionLS}, we inspect the weighted least-squares projection,
and point out why different strategies should be used when $r\gg1$ and $r\approx 1$.
Sections~\ref{section_effective_resistance} and \ref{section_fixed_increments} introduce and analyze the sampling Algorithms~\ref{algo_effective_resistance} and \ref{algo_fixed_increments}. They are independent from one another, up to the shared use of a few formulas, and the various estimates of Theorem~\ref{main_theorem} are proved separately.
Finally, {\cnew we discuss some numerical aspects of the presented algorithms in Section~\ref{section_numerical_aspects}, and provide numerical illustrations in Section~\ref{sec:numerical_illustration}}.

\section{Least-squares}
\label{SectionLS}
Let $\vp=(\vp_1,\dots,\vp_n)$ be an orthonormal basis of $V_n$ in $L^2(X,\mu)$. For any $x\in X$, we consider $\vp(x)$ as a vector in $\C^n$, denote $\vp(x)^*$ its conjugate transpose, and $|\vp(x)|^2=\vp(x)^*\vp(x)$ its 
squared euclidian norm.
Observe that, given any $n\times n$ matrix $\bM$,
\be
\int_{X} \vp(x)^*\bM\vp(x)d\mu(x)= \Tr\(\bM\int_X\vp(x)\vp(x)^*d\mu(x)\)=\Tr(\bM).
\label{orthonormality}
\ee
Adopting the formalism from \cite{KUV, MU, NSU}), we let ${\bf f}=(\sqrt{s_i}f(x_i))_{1\leq i \leq m}\in \C^m$ be the measurement vector, and
\[
\bL=(\sqrt{s_i}\vp_j(x_i))_{1\leq i \leq m, 1\leq j \leq n}\in \C^{m\times n}
\]
the collocation matrix. Then the weighted least-square estimator writes
\[
P_n^m f=\sum_{j=1}^na_j\vp_j,\quad a=\bL^+{\bf f},
\]
where $\bL^+=(\bL^*\bL)^{-1}\bL^*$ stands for the Moore-Penrose pseudo-inverse of $\bL$. In the sequel, we will make sure that $\bL$ has full column
rank $n$, so that the Gram matrix
\[
\bA_m:=\bL^*\bL=\sum_{i=1}^m s_i\vp(x_i)\vp(x_i)^*
\]
is indeed invertible.
Denoting $g=f-P_nf$ the optimal residual error and ${\bf g}=(\sqrt{s_i}g(x_i))_{1\leq i \leq m}$ the associated vector, the least-squares error decomposes as
\be
\|f-P_n^mf\|_{L^2}^2=\|f-P_nf\|_{L^2}^2+\|P_nf-P_n^mf\|_{L^2}^2=\|g\|_{L^2}^2+\|P_n^mg\|_{L^2}^2.
\label{Pythagoras}
\ee
There are two possible strategies for bounding $\|P_n^mg\|_{L^2}^2$:
either we use
\be
\|P_n^mg\|_{L^2}^2=|\bA_m^{-1}\bL^*{\bf g}|^2\leq \|\bA_m^{-1}\|_{2\to 2}^2|\bL^*{\bf g}|^2 = \lambda_{\min}(\bA_m)^{-2}|\<\vp,g\>_m|^2,
\label{LS_first_option}
\ee
where $\<\vp,g\>_m := (\<\vp_j,g\>_m)_{1\leq j \leq n}\in\C^n$, or we bound it by
\be
\|P_n^mg\|_{L^2}^2=|\bL^+{\bf g}|^2\leq \|\bL^+\|_{2\to 2}^2|{\bf g}|^2=\|\bA_m^{-1}\|_{2\to 2}\|g\|_m^2=\lambda_{\min}(\bA_m)^{-1}\|g\|_m^2.
\label{LS_second_option}
\ee

The first approach is expected to give better estimates when $m$ is much larger than~$n$, since in that case the discrete inner product $\<\cdot,\cdot\>_m$ weakly converges to the continuous inner product $\<\cdot,\cdot\>_{L^2}$, and as $g$ is orthogonal to $V_n$, the vector $\<\vp,g\>_m$ should be small. On the other hand, when $m$ is close to $n$, $\bL$ may be ill-conditioned, leading to small values of $\lambda_{\min}(\bA_m)$, thus favouring the second approach.
\newline

In both situations, one should choose the points $x_i$ and weights $s_i$ in order to control the smallest eigenvalue of $\bA_m$ from below.
This can be ensured through a greedy selection of points, based on the \emph{effective resistance} \cite{SS}
\[
\vp(x)^*(\bA-\ell \bI)^{-1}\vp(x)
\]
of a point $x\in X$ with respect to a hermitian matrix $\bA$, with $\ell$ a lower bound on the eigenvalues of~$\bA$, and~$\bI$ the identity matrix.
Intuitively, if the points $x_1,\dots,x_{i-1}$ are already fixed, the effective resistance of $x_{i}$ with respect to the partial Gram matrix
\[
\bA_{i-1}=\sum_{\iota =1}^{i-1} s_{\iota}\vp(x_{\iota})\vp(x_{\iota})^*
\] 
quantifies how close $\vp(x_{i})$ is to the eigenvectors of $\bA_{i-1}$ with small eigenvalues, helping the \emph{lower potential} $\Tr((\bA-\ell \bI)^{-1})$ to decrease when going from $\bA_{i-1}$ to $\bA_{i}$. The key idea in \cite{BSS} is to increase the lower barrier $\ell$ at each step, without increasing the potential more than it had decreased when adding sample $x_i$. Keeping the potential bounded ensures that all updates of $\ell$ are of the same size $\delta$, and its trace form is chosen to allow rank-one matrices $s_i\vp(x_i)\vp(x_i)^*$ to decrease the potential, independently of the distribution of eigenvalues of~$\bA_{i-1}$. In this way, we obtain the desired bound $\lambda_{\min}(\bA_m)>\ell_m$ at the end of the algorithm.

{\cnew As in \cite{BSS}, it is also possible to ensure an upper bound $\lambda_{\max}(\bA_m)< u_m$ on the eigenvalues of~$\bA_m$, by considering the \emph{upper potential} ${\Tr((u\bI - \bA)^{-1})}$. Although $\lambda_{\max}(\bA_m)$ does not explicitly appear in the least-squares formulation, it is worth noticing that the ratio $u_m/\ell_m$ controls the condition number $\lambda_{\max}(\bA_m)/\lambda_{\min}(\bA_m)$ of $\bA_m$, which determines the computational cost of solving the least-squares problem via an iterative solver, as well as its robustness to numerical error. This is discussed in, e.g., Section 5.3.4 of \cite{ABW}.

Nevertheless, we chose not to include this upper potential, because it would have prevented us from achieving the critical sampling budget $m\approx n$. Heuristically, when we use both upper and lower potentials, the vector $\vp(x_i)$ should increase the smallest eigenvalues of $\bA_{i-1}$, while staying far from the eigenvectors with largest eigenvalues. Due to the second condition, the lower barrier can only increase half as much as it could without it, resulting in the constraint $m\geq 2n$. In addition, we still achieve an a posteriori upper bound, see \eqref{lambdamax_Am}, which is not particularly sharp, but sufficient for our purposes. Lastly, a full inversion of matrix $\bA_m$, with complexity $\cO(n^3)$, remains reasonable in view of the moderate values of $n$ imposed by the sampling costs, see Section~\ref{section_numerical_aspects}.
}

\section{Random sampling by effective resistance}
\label{section_effective_resistance}

We first consider an approach combining \eqref{Pythagoras} and \eqref{LS_first_option} to obtain the first estimate \eqref{main_equation_r_large} in Theorem~\ref{main_theorem}. To observe the appropriate decay as $r\to \infty$, we impose that the vector
\[
\<\vp,g\>_m=\sum_{i=1}^m s_i\vp(x_i)\overline{g(x_i)}
\]
is an unbiased estimator of $\<\vp,g\>_{L^2}=0$, by taking weights $s_i$ inversely proportional to the sampling density of $x_i$. If the points $x_i$ are independent, the euclidian norm $|\<\vp,g\>_m|^2$ is therefore bounded in expectation by the sum of the variances of each term, of the form
\[
\E(s_i^2|\vp(x_i)|^2|g(x_i)|^2)\propto \int_X s_i| \vp(x_i)|^2|g(x_i)|^2 d\mu(x_i)\leq \big\|s_i|\vp(x_i)|^2\big\|_{L^\infty}\|g\|_{L^2}^2.
\]
As a result, to control $\ell_m$ from below and these variances from above, the appropriate sampling density for point $x_i$ is a combination of the effective resistance mentioned earlier 
and the so-called \emph{Christoffel function}
$|\vp(x)|^2=\sum_{j=1}^n|\vp_j(x)|^2$.

Algorithm \ref{algo_effective_resistance} {\cnew is inspired by
the deterministic procedure from \cite{BSS} and 
its randomization presented in \cite{LS_polynomial}, which has already been applied to least-squares recovery in \cite{CP}.
It} takes as inputs the parameters
\[
\eps \in (0,1), \qquad \eta =\frac{\eps}{1-\eps}\in (0,\infty), \qquad \text{and}\qquad \gamma\geq0,
\]
which respectively influence the size of the barrier increments $\delta_i=\ell_i-\ell_{i-1}$, the weights~$s_i$,
and the balance between effective resistance and Christoffel function in the sampling density. 

\begin{algorithm}
\begin{algorithmic}
{\cnew
\STATE{{\bf Input:} probability measure $\mu$, parameter $\eps \in (0,1)$, and $\gamma\geq0$.}
}
\STATE{{\cnew $\bA_0={\bf0}\in \R^{n\times n}$}, $\ell_0=-n$}
\FOR{$i=1,\dots,m$}
\STATE{Let $\bY_i= (\bA_{i-1}-\ell_{i-1} \bI)^{-1}$}
\STATE{Update $\ell_{i}=\ell_{i-1}+\delta_i$,\quad where \quad 
$\delta_i:={\eps}/{\Psi_i}$\quad with \quad $\Psi_i=\Tr(\bY_i)+\gamma$}
\STATE{Let $\bZ_i=(\bA_{i-1}-\ell_{i} \bI)^{-1}$}
\STATE{Let $ \rho_i:x\mapsto \vp(x)^*(\bZ_i +\gamma\bI/n)\vp(x)$}
\STATE{Draw $x_{i}$ from probability measure $\frac{\rho_i(x)}{\Xi_i} d\mu(x)$, 
\quad where \quad $\Xi_i=\Tr({\cnew \bZ_i})+\gamma$}
\STATE{Let $s_{i}={\eta}/{\rho_i(x_i)}$}
\STATE{Update $\bA_{i}= \bA_{i-1}+s_{i}\vp(x_{i})\vp(x_{i})^*$}
\ENDFOR
{\cnew
\STATE{{\bf Output:} sample $\{x_1,\dots,x_m\}$, weights $\{s_1,\dots,s_m\}$,
and matrix $\bA_m$.}
}
\end{algorithmic}
\caption{Random sampling by effective resistance}
\label{algo_effective_resistance}
\end{algorithm}

For the sake of the analysis, we also define
\[
\begin{array}{c}
\bY_{m+1}=(\bA_m-\ell_m {\cnew \bI})^{-1},\quad
\Psi_{m+1}=\Tr(\bY_{m+1})+\gamma, \\
\displaystyle
\delta_{m+1}=\frac{\eps}{\Psi_{m+1}}\quad \text{and}\quad 
\ell_{m+1}=\ell_m+\delta_{m+1},
\end{array}
\]
as would have been done at iteration $i=m+1$.

\begin{lemma}
The algorithm is {\cnew almost surely} well defined and
$\bA_i\succ \ell_{i+1}\bI$ for all $0\leq i \leq m$, where $\succ$ stands for the Loewner order between positive semi-definite matrices.
\label{lemma_potential_control}
\end{lemma}

\begin{proof}
We proceed by induction on $i$. At initialization $\bA_0\succ - n \bI = \ell_0 \bI$.
For $i \in \{1,\dots,m\}$ fixed, assume that {\cnew $\bA_{i-1}\succ \ell_{i-1}\bI$}. This 
implies that $\bY_i$ is well-defined, positive definite, and
{\cnew
\be
\delta_i \bI = \frac{\eps\bI}{\Tr(\bY_i)+\gamma} 
{\cnew \preceq }
\frac{\eps\bI}{\lambda_{\max}(\bY_i)}
= {\lambda_{\min}(\bA_{i-1}-\ell_{i-1} \bI)}\,\eps\bI \prec \bA_{i-1}-\ell_{i-1} \bI.
\label{boundOnDelta}
\ee
Hence $\bA_{i-1}\succ (\ell_{i-1} +\delta_i) \bI = \ell_{i}\bI$}, so $\bZ_{i}$ is well-defined, 
positive definite, and by identity \eqref{orthonormality},
\[
\int_X\rho_i(x)\,d\mu(x)
=\Tr(\bZ_i+\gamma\bI/n)=\Xi_i,
\]
proving that $\rho_i d\mu/\Xi_i$ is indeed a probability density.
{\cnew  Finally $s_i > 0$ almost surely, so $\bA_{i} \succcurlyeq \bA_{i-1} \succ \ell_{i}\bI$. This 
completes the induction, and a last iteration shows that $\bA_m\succ \ell_{m+1}\bI$, concluding the proof.}
\end{proof}

Contrarily to most variations on the algorithm from \cite{BSS}, no upper potential is used to bound the eigenvalues of $\bA_i$ from above. Instead, parameter $\gamma$ provides a lower bound on $\rho_i$ in terms of the Christoffel function $\vp(x)^*\vp(x)$, which turns into an upper bound on the norm of rank-one terms $s_i\vp(x_i)\vp(x_i)^*$, and therefore on the norm of their sum $\bA_m$.

To bound the eigenvalues of $\bA_i$ from below, we use the lower barrier $\ell_{i+1}$, which should increase at each step, at a speed $\delta_i$ controlled by the lower potential $\Tr(\bY_i)$. As the densities $\rho_i$ are positive, the sample $(x_1,\dots,x_m)$ could be in any part of $X^m$,
so there cannot be any positive deterministic lower bound on $\lambda_{\min}(\bA_m)$. However,
the following lemma proves a monotonicity property in expectation, similar to Lemmas~4.4 in \cite{LS_polynomial} and \cite{LS_exponential}, themselves inspired by Lemmas 3.3 and 3.4 in \cite{BSS}.

\begin{lemma}
The sequence 
$(\E(\Tr(\bY_i))_{1\leq i \leq m+1}$ is non-increasing.
\label{lemma_monotonicity}
\end{lemma}
\begin{proof}
Let $1\leq i \leq m$,
and denote $u_i=\sqrt{s_i}\vp(x_i)$.
Applying the Sherman-Morrison formula yields
\be
\bY_{i+1}=(\bZ_i^{-1}+u_iu_i^*)^{-1}
=\bZ_i-\frac{\bZ_iu_iu_i^*\bZ_i}{1+u_i^*\bZ_iu_i}.
\label{Sherman-Morrison}
\ee
By design of the sampling density and weights,
\be
\label{eq_uZu}
u_i^*\bZ_iu_i=s_i\vp(x_i)^*\bZ_i\vp(x_i)\leq s_i\rho_i(x_i)=\eta,
\ee
so by definition of {\cnew $\eta=\eps/(1-\eps)$, which implies $1+\eta={\eta}/{\eps}$, we obtain}
\begin{equation}
\bY_{i+1}- \bZ_i\preccurlyeq-\frac{\bZ_iu_iu_i^*\bZ_i}{1+\eta}= -\frac{\eps}{\eta}s_i \bZ_i\vp(x_i)\vp(x_i)^*\bZ_i.
\label{YminusZ}
\end{equation}
On the other hand, $\bY_i$ and $\bZ_i$ commute because their inverses do, and we can write
\[
\bZ_i-\bY_i=(\bY_i^{-1}-\bZ_i^{-1})\bY_i\bZ_i=\delta_i\bY_i\bZ_i.
\]
Combining this identity to the previous estimate and taking traces gives
\[
\Tr(\bY_{i+1})-\Tr(\bY_i)=\delta_i\Tr(\bY_i\bZ_i)-\frac{\eps}{\eta}s_i \vp(x_i)^*\bZ_i^2\vp(x_i).
\]
Observe that $\bY_i$ and $\bZ_i$ only depend on the
samples $x_1,\dots,x_{i-1}$. As a result, if these samples are fixed,
taking the expectation $\E_{x_i}=\E(\,\cdot\,|x_1,\dots,x_{i-1})$ with respect to point $x_i$, we arrive at 
\begin{equation}
\begin{aligned}
\E_{x_i}\big(\Tr(\bY_{i+1})\big) -\Tr(\bY_i)
&= \delta_i \Tr(\bY_i\bZ_i)- \frac{\eps}{\eta}\int_X \frac{\eta}{\rho_i(x_i)}\vp(x_i)^*\bZ_i^2\vp(x_i) \frac{\rho_i(x_i)}{\Xi_i}d\mu(x_i)\\
&=\eps\(\frac{\Tr(\bY_i\bZ_i)}{\Tr(\bY_i)+\gamma}
-\frac{\Tr(\bZ_i^2)}{\Tr(\bZ_i)+\gamma}\),
\end{aligned}
\label{expectation_of_the_trace}
\end{equation}
{\cnew where we have used identity \eqref{orthonormality}}. Since
\be
\bZ_i = \bY_i + \delta_i \bY_i\bZ_i\quad\text{and}\quad  
\bZ_i^2 =\bY_i\bZ_i + \delta_i \bY_i\bZ_i^2,
\label{YZ}
\ee
we have by Cauchy-Schwarz inequality
\be
 \Tr(\bY_i\bZ_i) \Tr(\bZ_i) - \Tr(\bY_i) \Tr(\bZ_i^2) 
=\delta_i\( \Tr(\bY_i \bZ_i)^2 - \Tr(\bY_i) \Tr(\bY_i \bZ_i^2) \) \leq0.
\label{Cauchy_Schwarz}
\ee
Moreover,
{\cnew $\Tr(\bY_i\bZ_i) - \Tr(\bZ_i^2) 
= -\delta_i \Tr(\bY_i \bZ_i^2) = -\delta_i \Tr( \bZ_i \bY_i \bZ_i)<0$. Hence} 
the right-hand side in \eqref{expectation_of_the_trace} is negative. 
Taking an expectation over the previous samples $x_1,\dots,x_{i-1}$ in turn implies that 
$\E \big(\Tr(\bY_{i+1}) \big)$ is smaller than $ \E\big(\Tr (\bY_i)\big)$, which concludes
{\cnew the proof}.
\end{proof}

One main difference with \cite{LS_polynomial} and \cite{LS_exponential} is the choice of different matrices $\bY_i$ and $\bZ_i$ for updating the lower barrier $\ell_i$ and selecting a new point $x_i$. This comes at the expense of the refined analysis \eqref{YZ}, \eqref{Cauchy_Schwarz} to show that {\cnew \eqref{expectation_of_the_trace}} is non-positive. Moreover, in view of \eqref{boundOnDelta}, this changes the sampling density at most by a factor $1/(1-\eps)$.

{\cnew However, if we had used $\bY_i$ instead of $\bZ_i$ in the sampling density $\rho_i$, in
other words if we had defined $\rho_i(x)= \vp(x)^*(\bY_i +\gamma\bI/n)\vp(x)$
and $\Xi_i =\Tr({\bY_i})+\gamma$, we  
would have obtained
\[
u_i^*\bZ_i u_i\leq u_i^* \frac{\bY_i}{1-\eps}u_i \leq \frac{s_i\rho_i(x_i)}{1-\eps}=\frac{\eta}{1-\eps}
\]
instead of \eqref{eq_uZu}, where the first inequality comes from the fact that $\bY_i^{-1} - \bZ_i^{-1} \leq \eps \bY_i^{-1}$ in view of \eqref{boundOnDelta}.
In order to recover the same right-hand side as before in 
inequality \eqref{YminusZ}, one would then need to assume}
\[
\frac{1}{1+\frac{\eta}{1-\eps}}\geq \frac{\eps}{\eta}\quad \Longleftrightarrow \quad 1+\frac{\eta}{1-\eps}\leq \frac{\eta}{\eps} \quad \Longleftrightarrow \quad \eta\geq \frac{1}{\frac{1}{\eps}-\frac{1}{1-\eps}}=\frac{\eps(1-\eps)}{1-2\eps}.
\]
So this simpler approach {\cnew requires} $\eps<\frac{1}{2}$, which is not possible unless 
$r>2$, and even in that case our approach gives slightly better estimates.
\newline

The control on the lower potentials in expectation from Lemma~\ref{lemma_monotonicity} yields a lower bound on the eigenvalues of $\bA_m$ in probability.

\begin{proposition}
\label{prop_Am_algo1}
Let $m\geq n$, $p\in(0,1)$, and define $r=(m+1)/n$. The random matrix $\bA_m$ 
generated by Algorithm \ref{algo_effective_resistance} {\cnew satisfies}
\[
\bA_m \succcurlyeq  
\alpha \bI,\qquad \alpha=n\(\frac{\eps r}{1/p+\gamma}-1\),
\]
with probability at least $1-p$.
\end{proposition}

\begin{proof}
In view of the monotonicity property established in
Lemma \ref{lemma_monotonicity}, and thanks to the initialization $\ell_0=-n$, we have
\[
\E \big(\Tr(\bY_i)\big)\leq \Tr(\bY_1)=1,\qquad 1\leq i\leq m+1.
\]
By Markov's inequality
\[
\P\(\sum_{i=1}^{m+1} \Tr(\bY_i) > \frac{m+1}{p}\) \leq p\,\frac{\E(\sum_{i=1}^{m+1} \Tr(\bY_i))}{m+1}\leq p.
\]
As $\Psi_i=\Tr(\bY_i)+\gamma$,
we deduce that with probability at least $1-p$, it holds
\[
\sum_{i=1}^{m+1} \Psi_i \leq (m+1)\(\frac{1}{p}+\gamma\).
\]
Together with the inequality 
$\(\sum_{i=1}^{m+1} \Psi_i\)\(\sum_{i=1}^{m+1} \Psi_i^{-1}\)\geq (m+1)^2$, this implies
\[
\sum_{i=1}^{m+1} {\Psi_i}^{-1}\geq \frac{m+1}{1/p + \gamma}
=\frac{rn}{1/p + \gamma},
\]
and therefore
\[
\ell_{m+1}
=\ell_0+\sum_{i=1}^{m+1} \frac{\eps}{\Psi_i} 
\geq - n + \frac{rn \eps}{1/p + \gamma}
= \alpha,
\]
with probability at least $1-p$. Since $\bA_m\succ \ell_{m+1} \bI$, see 
Lemma \ref{lemma_potential_control}, the proof is complete.
\end{proof}

\begin{remark}
The lower bound $\bA_m \succcurlyeq \alpha I$ can be rewritten
as a frame inequality
\[
c^*\bA_m c=\sum_{i=1}^m s_i |\<c,\vp(x_i)\>|^2\geq \alpha |c|^2,\qquad c\in\C^n,
\]
or as a Marcinkiewicz-Zygmund inequality: for all $v\in V_n$,
$\|v\|_m^2\geq \alpha \|v\|_{L^2}^2$.
All three versions have been extensively used in the {\cnew literature} for emphasizing 
the relations with subsampling of frames and discretization of continuous norms, 
see \cite{FS, NOU2011, NOU2016, LT}. {\cnew Note that our approach is not performing a subsampling in general, although it can be interpreted as such in the context of Remark~\ref{rk_finite_X}.}
\end{remark}

We are now ready for the proof of the first statement of Theorem~\ref{main_theorem}. Let $E$ be the event of probability at least $1-p$ where $\bA_m \succcurlyeq \alpha \bI$, 
{\cnew for $p\in (0,1)$ and $\alpha>0$ given by Proposition 
\ref{prop_Am_algo1}}, and define
\begin{equation}
\tilde f= P_n^mf|E
\label{conditional_WLS}
\end{equation}
as the weighted least-squares projection, conditioned to event $E$. In practice, this amounts to relaunching Algorithm \ref{algo_effective_resistance} until $E$ occurs, and computing the weighted least-squares projection of $f$ with the last sample.

\begin{proof}[Proof of Theorem \ref{main_theorem}, equation \eqref{main_equation_r_large}]
As $V_n\subset L^2$, we can assume that $f$ is also in $L^2$, otherwise the right-hand side would be infinite.
Similar to \eqref{Pythagoras}, by Pythagoras theorem,
\begin{align*}
\E( \|f-\tilde f\|_{L^2}^2)&=\|f-P_nf\|_{L^2}^2+ \E(\|\tilde f-P_nf\|_{L^2}^2)\\
&=\|g\|_{L^2}^2+\E(\|P_n^m f-P_nf\|_{L^2}^2|E)\\
&=\|g\|_{L^2}^2+\E(\|P_n^m g\|_{L^2}^2|E).
\end{align*}
where $g=f-P_nf$. Combining this with equation \eqref{LS_first_option}, together with the definition of $E$, yields
\[
\E( \|f-\tilde f\|_{L^2}^2) \leq \|g\|_{L^2}^2 + \alpha^{-2}\,\E(|\<\vp,g\>_m|^2|E).
\]
We now use Proposition~\ref{prop_Am_algo1}:
\[
\E(|\<\vp,g\>_m|^2|E)=\frac{\E(|\<\vp,g\>_m|^2\chi_E)}{\P(E)}\leq \frac{1}{1-p}\E(|\<\vp,g\>_m|^2).
\]
Develop the discrete inner products $\<\cdot,\cdot\>_m$ as
\[
\E(|\<\vp,g\>_m|^2)=\E\( \sum_{k=1}^n |\<\vp_k,g\>_m|^2\)
=\sum_{k=1}^n \E \left(\sum_{i=1}^{m} s_i \overline{\vp_k(x_i)}{g(x_i)} \sum_{j=1}^{m} s_j {\vp_k(x_j)}\overline{g(x_j)}\right).
\]
When $i > j$, we have that
\[
\E_{x_i}\(s_i \overline{\vp_k(x_i)}{g(x_i)}\)
= \frac{\eta}{\Xi_i} \int_{X} \overline{\vp_k(x)} g(x) d\mu(x)
=0,
\]
since $g$ is orthogonal to $V_n$, so the corresponding term is zero. The same 
holds true for $i<j$, by exchanging $i$ and $j$. We are only left with the diagonal terms
\[
\E(|\<\vp,g\>_m|^2)
= \E \( \sum_{i=1}^{m} s_i^2 | \vp(x_i)|^2  |g(x_i)|^2\)
\leq \frac{n\eta}{\gamma} \E \( \sum_{i=1}^{m} s_i |g(x_i)|^2\)
= \frac{n \eta}{\gamma} \frac{m \eta}{\cnew \Xi_i }
\|g\|_{L^2}^2,
\]
where {\cnew we have first used the uniform bound $s_i|\vp(x_i)|^2\leq n\eta/\gamma$,
then integrated over $x_i$ using $s_i\rho(x_i)=\eta$,
both identities following from the definition of $s_i$ in Algorithm~\ref{algo_effective_resistance}}. Since {\cnew$\Xi_i\geq \gamma$}, $m\leq rn$ and 
$\eta=\eps/(1-\eps)$, we deduce that 
\be
\begin{aligned}
\E( \|f-\tilde f\|_{L^2}^2) 
&\leq \(1+ \frac {1}{\alpha^2 (1-p)} \frac{rn^2\eps^2}{\gamma^2 (1-\eps)^2}\)\|g\|_{L^2}^2 
\\&=
\(1+\frac r {1-p} 
\frac {1}{\beta (\eps,\gamma)^2}
\) \|g\|_{L^2}^2,
\end{aligned}
\label{bound_all_params}
\ee
{\cnew where 
\[
\beta (\eps,\gamma) =  \frac{\alpha\gamma(1-\eps)}{n\eps}=\gamma\(\frac{r\eps}{1/p+\gamma}-1\)\(\frac{1}{\eps}-1\)
\]
given the definition of $\alpha$ in Proposition \ref{prop_Am_algo1}.
For $p$ fixed, maximizing $\beta (\eps,\gamma)$
over $\eps \in (0,1)$ and $\gamma >0$ yields conditions}
\[
\frac{1}{p}+\gamma=r\eps^2\quad\text{and}\quad \(\frac{1}{p}+\gamma\)^2=\frac{r\eps}{p},
\]
{\cnew which are equivalent to $r\eps^3 = {1}/{p}$ and $\gamma = r\eps^2(1-\eps)$, and the maximum value of $\beta(\eps,\gamma)$
is $r (1-\eps)^3$.
Letting $p = r^{-1/4} \in (0,1)$, and thus $\eps = r^{-1/4} \in (0,1)$ and $\gamma = r^{1/2} - r^{1/4} >0$,
we obtain the upper bound in \eqref{main_equation_r_large}.
Note that the first optimality condition above automatically implies 
$\alpha = n(\eps^{-1}-1) >0$, making the proof consistent.
}
\end{proof}

Notice that, with this last choice
$p=\eps=r^{-1/4}$, $\gamma=(1-\eps)/\eps^2$ and $\alpha=n(1-\eps)/\eps$. In particular,
as $r$ tends to infinity, the probability of failure $1-p$ in Proposition~\ref{prop_Am_algo1}
goes to zero, and $\gamma$ goes to infinity, meaning that the sampling density gets closer 
to the Christoffel function, as in \cite{CM}.

\begin{remark}
In the example where $V_n$ consists of piecewise constant functions on a fixed partition of $X$ into pieces of equal measure, consider $f$ such that $f-P_nf$ is a realisation of a white noise of variance $1$. It is easily seen that the optimal approximation error is the Monte-Carlo rate $1+1/r+o(1/r)$, with $\tilde f$ made of averages over $r$ random samples in each piece. In that sense, Algorithm~\ref{algo_effective_resistance} achieves the optimal decay rate of the error, similar to \cite{CM,HNP}, but with a remainder independent of $n$.
\end{remark}

\section{Refined randomized sampling algorithm}
\label{section_fixed_increments}

We now seek to optimize the approximation strategy in the regime $r\approx 1$.
Although Algorithm~\ref{algo_effective_resistance} works as soon as $r\geq 1$, several improvements can be performed for small values of~$r$.
First, using \eqref{LS_second_option} instead of \eqref{LS_first_option} reduces the factor $\eta^2/\alpha^2$ to $\eta/\alpha$ in bound \eqref{bound_all_params}. Secondly, bounding the weights $s_i$ by a multiple of $1/\gamma$ becomes a crude estimate when $r\to 1$, since $\gamma\to 0$ in that case.
Lastly, the use of Markov's inequality in Proposition~\ref{prop_Am_algo1} comes at the expense of a factor $1/(1-p)$, which grows as $r$ gets close to $1$; to avoid the last issue, we would like the lower barrier $\ell$ to grow in a steady, deterministic fashion. This prevents the sample points $x_i$ from being drawn anywhere in the domain $X$.

For all these reasons, we consider Algorithm \ref{algo_fixed_increments}, with inputs 
$\delta \in (0,1)$, the size of lower barrier increments, 
and $\kappa \in [0,1]$, a parameter balancing the $L^2$ and $L^\infty$ settings.
Note that we will always fix $m$ beforehand, and 
set $\delta = 1/\sqrt r$, where $r=m/(n-1)$.

Again, we define $\bY_{m+1}=(\bA_m-\ell_m\bI)^{-1}$, as we would have done at iteration $i=m+1$.
\newline
\bigskip

\begin{algorithm}
\begin{algorithmic}
{\cnew
\STATE{{\bf Input:} probability measure $\mu$, parameter $\delta \in (0,1)$, and $\kappa \in [0,1]$.}
}
\STATE{{\cnew $\bA_0={\bf0}\in \R^{n\times n}$}, $\ell_0=-n$}
\FOR{$i=1,\dots,m$}
\STATE{Let $\bY_i= (\bA_{i-1}-\ell_{i-1} \bI)^{-1}$}
\STATE{Update \cnew $\ell_{i}= \ell_{i-1}+\delta$}
\STATE{Let $\bZ_i=(\bA_{i-1}- \ell_{i} \bI)^{-1}$}
\STATE{Let $\bW_i=\big(\Tr(\bZ_i)-\Tr(\bY_i)\big)^{-1} \bZ_i^2 - \bZ_i$ \quad and \quad $w_i:x\mapsto \vp(x)^* \bW_i \vp(x) $}
\STATE{Define $R_i(x)=w_i(x) {\mathbf 1}_{w_i(x) \geq \kappa\frac{1-\delta}{\delta}}$}
\STATE{Draw $x_{i}$ from probability density  $\frac{R_i(x)}{\Gamma_i}d\mu(x)$, \quad where \quad
$\Gamma_i = \int_{X} R_i(x)d\mu(x)$}
\STATE{Let $s_{i}=1/R_i(x_{i}) = \cnew 1/w_i(x_{i}) $}
\STATE{Update $\bA_{i}= \bA_{i-1}+s_{i}\vp(x_{i})\vp(x_{i})^*$}
\ENDFOR
{\cnew
\STATE{{\bf Output:} Sample $\{x_1,\dots,x_m\}$, weights $\{s_1,\dots,s_m\}$,
and matrix $\bA_m$.}
}
\end{algorithmic}
\caption{Random sampling with fixed increments of the barrier}
\label{algo_fixed_increments}
\end{algorithm}

\begin{lemma}
The algorithm is well-defined  and $\Tr(\bY_i)=1$
for $1\leq i \leq m+1$.
\end{lemma}

\begin{proof}
We use an induction on $i$. First, note that $\bY_1 = \bI/n $ hence $\Tr(\bY_1)=1$. 
Let $i\in\{1,\dots,m\}$,
and assume that $\bY_i$ is positive definite with $\Tr(\bY_i)=1$.
We have
\[
\lambda_{\min}(\bY_i^{-1})= \lambda_{\max}(\bY_i)^{-1} \geq (\Tr(\bY_i))^{-1}= 1 > \delta,
\]
Therefore {\cnew $\bZ_i = (\bY_i^{-1} - \delta \bI)^{-1}$ is well defined and 
positive definite, and $\bW_i$ is also well defined since $\bZ_i \succ \bY_i$.}
We then need to show that 
$R_i(x)$ is not zero a.e., in other words that
$w_i(x) \geq \kappa (1-\delta)/\delta$ on a set of positive mesure. 
To this end, we will rely on an averaging argument by showing that 
$\int_{X} w_id\mu$, which is equal to $\Tr(\bW_i)$ according to \eqref{orthonormality}, is larger than $(1-\delta)/\delta$.
Recalling \eqref{YZ} and \eqref{Cauchy_Schwarz},
we observe as in Claim 3.6 of \cite{BSS}
\be
\begin{array}{ll}
\Tr (\bW_i)&=\displaystyle \frac{\Tr (\bZ_i^2)}{\Tr (\bZ_i)-\Tr (\bY_i)} -\Tr (\bZ_i)\\
&=\displaystyle \frac{1}{\delta}\(\frac{\Tr(\bZ_i^2)}{\Tr(\bY_i\bZ_i)}-\frac{\Tr(\bZ_i)}{\Tr(\bY_i)}\)+\frac{1-\delta}{\delta}\Tr(\bZ_i)> \frac{1-\delta}{\delta},
\end{array}
\label{lowerbound_TrW}
\ee
where we used the induction hypothesis $\Tr (\bY_i)= 1$,
and its immediate consequence ${\Tr(\bZ_i)>1}$.
Finally, using the Sherman-Morrison formula as in \eqref{Sherman-Morrison}, 
with $u_i=\sqrt{s_i}\vp(x_i)$,
\[
\Tr(\bY_{i+1})
=\Tr(\bZ_i)-\frac{u_i^*\bZ_i^2u_i}{1+u_i^*\bZ_iu_i}=\Tr(\bZ_i)-\frac{\vp(x_i)^*\bZ_i^2\vp(x_i)}{w_i(x_i)+\vp(x_i)^*\bZ_i\vp(x_i)}=\Tr(\bY_i)=1,
\]
which concludes the induction.
\end{proof}

{\cnew
\begin{remark}
\label{lam_max_W}
The symmetric matrix $\bW_i$ is not 
necessarily positive semi-definite, but we have a framing on its maximal eigenvalue
\[
0<\frac{\Tr(\bW_i)}{n}\leq \lambda_{\max}(\bW_i)\leq\frac{\lambda_{\max}(\bZ_i^2)}{\Tr(\bZ_i)-\Tr(\bY_i)}\leq\frac{\Tr(\bZ_i^2)}{\delta\Tr(\bY_i\bZ_i)}\leq\frac{1}{\delta(1-\delta)},
\]
since $\bZ_i \succ \bY_i \succ (1-\delta) \bZ_i$. This property will be useful in the discussion of Section~\ref{section_numerical_aspects}.
\end{remark}
}

As $\ell_i=i\delta-n$, we easily get a lower bound on the eigenvalues of $\bA_m$.
\begin{proposition}
\label{prop_Am_algo2}
Let $m\geq n$ and define $r=m/(n-1)$. The random matrix $\bA_m$ 
generated by Algorithm \ref{algo_fixed_increments} verifies
\[
\bA_m \succcurlyeq \alpha \bI\quad a.s.,\qquad \text{with}\quad \alpha=(n-1) (\delta r - 1 ).
\]
\end{proposition}

\begin{proof}
After the last iteration of the algorithm, we have
\[
\lambda_{\min}(\bA_m)=\ell_m+(\lambda_{\max}(\bY_{m+1}))^{-1}\geq m\delta-n+(\Tr(\bY_{m+1}))^{-1}= (\delta r-1)(n-1).
\]
\end{proof}

We now have all the tools for the proof of the main theorem.
\begin{proof}[Proof of Theorem \ref{main_theorem}, equation \eqref{main_equation_r_small}]
Introducing the set
\[
E_i = \left\{x\in {X},\; w_i(x)\geq \kappa \frac{1-\delta}\delta\right\},
\]
observe that
\be
\label{eq_Gamma}
\begin{array}{ll}
\Gamma_i &= \displaystyle\int_{E_i} R_i d\mu =\displaystyle\int_{X} w_i(x) d\mu(x) - \int_{E_i^c} w_i(x)d\mu(x) \\
& \geq  \displaystyle \Tr(\bW_i) - \int_{X} \kappa \frac{1-\delta}{\delta}d\mu \geq (1-\kappa)\frac{1-\delta}{\delta}.
\end{array}
\ee
Taking the expectation $\E_{x_i}$ over point $x_i$, with $x_1,\dots,x_{i-1}$ fixed, we compute
\[
\E_{x_i}(s_{i} |g(x_{i})|^2 ) = \int_{E_i} \frac{1}{R_i(x)} |g(x)|^2 \frac{R_i(x)} {\Gamma_i} d\mu(x)
\leq \frac{\|g\|_{L^2}^2}{\Gamma_i}
\leq 
\frac 1 {1-\kappa}
\frac{\delta}{1-\delta} \|g\|_{L^2}^2,
\]
where $g=f-P_nf\in L^2(X,\mu)$.
The same holds for all $1\leq i \leq m$, and consequently 
\[
\E\(\|g\|_m^2\)\leq 
\frac m {1-\kappa}
\frac{\delta}{1-\delta}
\|g\|_{L^2}^2,
\]
which together with \eqref{LS_second_option} and Proposition~\ref{prop_Am_algo2} gives
\[
\E(\|P_n^mg\|^2) 
\leq \alpha^{-1}\E\(\|P_n^mg\|_{m}^2\) 
{\cnew \leq \alpha^{-1}\E\(\| g\|_{m}^2\)}
\leq \displaystyle\frac{r}{m(\delta r-1)} 
\frac m {1-\kappa}
\frac{\delta}{1-\delta}
\|g\|_{L^2}^2.
\]
Taking $\delta={1}/{\sqrt r}$ and adding \eqref{Pythagoras}, 
we arrive at \eqref{main_equation_r_small}.
\end{proof}

\begin{remark}
For $\kappa=0$, the density of point $x_i$ is proportional to $w_i(x)_+=\max(w_i(x),0)$. A sampling density proportional to the positive part of a function also occurs in \cite{LS_exponential}, equation~(6). There, it is used to select points which increase the lower potential without increasing too much the upper potential.
\end{remark}

\begin{proof}[Proof of Theorem \ref{main_theorem}, equation \eqref{main_equation_uniform}]
As $V_n\subset L^\infty$, we can assume that $f$ is also in $L^\infty$, otherwise the right-hand side would be infinite.
We introduce the notation
\[
P_n^\infty f{\cnew \in}\argmin_{v\in V_n}\|f-v\|_{L^\infty}
\]
for a Chebyshev projection of $f$ onto $V_n$ with respect to the Banach norm $\|\cdot\|_{L^\infty}$, and take $g=f-P_n^\infty f$ the associated residual. As $1/{s_{i}}=w_i(x_{i})\geq \kappa{(1-\delta)}/{\delta}$,
it holds
\[
\|g\|_m^2=\sum_{i=1}^m s_i|g(x_i)|^2\leq \sum_{i=1}^m s_i\|g\|_{L^\infty}^2 \leq \frac{m\delta}{\kappa(1-\delta)}\|g\|_{L^\infty}^2\quad a.s.
\]
By a triangular inequality
\[
\|f-P_n^mf\|_{L^2}\leq \|f-P_n^\infty f\|_{L^2} +  \|P_n^m f-P_n^\infty f\|_{L^2} =\|g\|_{L^2}+ \|P_n^m g\|_{L^2}.
\]
We conclude by bounding the operator norm $\|P_n^m\|_{L^\infty\to L^2}$ through
\be
\|P_n^mg\|_{L^2}^2
\leq \alpha^{-1} \|P_n^mg\|_m^2
\leq \alpha^{-1}\|g\|_m^2
\leq \frac{r}{\delta r- 1} \frac{\delta}{ \kappa(1-\delta)}\|g\|_{L^\infty}^2.
\label{projection_norm_uniform}
\ee
As a result, Algorithm~\ref{algo_fixed_increments} with 
$\delta={1}/{\sqrt r}$ implies inequality \eqref{main_equation_uniform}.
\end{proof}

\begin{remark}
In the randomized setting, $\mu$ could be any positive measure on $X$.
On the contrary, in the uniform setting, we need $\mu$ to be a probability measure to bound $L^2$ norms by $L^\infty$ norms. 
Similar results would hold if $\mu$ was a positive measure of finite mass, with all $L^\infty$ norms multiplied by a factor $\mu(X)$, see for instance \cite{LT, T, KPUU}. 
\end{remark}

We turn to the proof of the last result of the introduction, which is stated in 
strong uniform norm to avoid technicalities.
\begin{proof}[Proof of Corollary~\ref{corollary_KW}] Denote $\mathcal L^\infty$ the set of bounded functions, and $\|f\|_{\mathcal L^\infty}:=\sup_{x\in X}|f(x)|$ the associated norm. As $V_n\subset \mathcal L^\infty$, we can assume that $f$ is also in $\mathcal L^\infty$, otherwise the right-hand side would be infinite. Taking $g=f-P_n^\infty f$ the $\mathcal L^\infty$-residual, we write as above
\[
\|f-P_n^mf\|_{\mathcal L^\infty}\leq \|f-P_n^\infty f\|_{\mathcal L^\infty} +  \|P_n^m f-P_n^\infty f\|_{\mathcal L^\infty} =\|g\|_{\mathcal L^\infty}+ \|P_n^m g\|_{\mathcal L^\infty}.
\]
For any probability measure $\mu$ on $X$,
consider $(\vp_1,\dots,\vp_n)$ an orthonormal basis of $V_n$ in $L^2=L^2(X,\mu)$. Then the supremum of the Christoffel function is, by a Cauchy-Schwarz inequality,
\[
K_n:=\sup_{x\in X}|\vp(x)|^2=\sup_{x\in X}\max_{c\in \R^n \setminus\{0\}}\frac{|\sum_{j=1}^nc_j\vp_j(x)|^2}{|c|^2}=\max_{v\in V_n \setminus \{0\}}\frac{\|v\|_{\mathcal L^\infty}^2}{\|v\|_{L^2}^2}.
\]
{\cnew By simply considering \eqref{projection_norm_uniform} in the case $m=n$, $\kappa=1$, and $\delta=1/\sqrt r\leq 1-1/2n$, we derive the inequality
$\|P_n^n g\|_{L^2} \leq (1-\delta)^{-1}\|g\|_{\mathcal L^\infty}\leq 2n \|g\|_{\mathcal L^\infty}$. As $P_n^n g\in V_n$, we obtain} 
\[
\|P_n^n g\|_{\mathcal L^\infty} \leq 
\sqrt{K_n}\|P_n^n g\|_{L^2}\leq 2n \sqrt{K_n} \|g\|_{\mathcal L^\infty}.
\]
We now rely on the main corollary of the paper \cite{KW}.
Since $X$ is compact, there exists a measure $\mu^*$ for which 
$K_n=n$.
Using such a measure in our framing, 
we obtain by application of the above 
$\|P_n^n g\|_{\mathcal L^\infty} \leq 2n\sqrt n \|g\|_{\mathcal L^\infty}$,
which concludes the proof.
\end{proof}

Let us briefly explain the result in \cite{KW} to our setting. Given 
$\phi = (\phi_1,\dots,\phi_n)$ a basis of $V_n$ and $\mu$
a probability measure over $X$, we consider the $n\times n$ 
symmetric positive semi-definite matrix $M(\mu) := \int_X \phi\phi^* d\mu$. Assuming 
this matrix is non-singular, $\vp:= {M(\mu)}^{-1/2} \phi$ is also basis 
of $V_n$ and is orthonormal with respect to $\mu$. In particular 
$\vp(x)^*\vp(x) = \phi^* {M(\mu)}^{-1} \phi $ is the 
associated Christoffel function. The optimal probability measures from 
\cite{KW} are characterized by
\[
\mu^*\in \argmin_{\mu\in \mathrm{Prob}(X)}\,\sup_{x\in X}\; \phi(x)^*M(\mu)^{-1}\phi(x)
=\argmax_{\mu\in \mathrm{Prob}(X)}\,\det\(M(\mu)\).
\]
These two equivalent extremum problems are difficult 
to solve for general compacts $X$ and spaces $V_n$, thus making 
the approach from Corollary~\ref{corollary_KW} unpractical. Nevertheless, it gives a theoretical bound on the approximation error achievable by this method.
We refer to \cite{Bos} for an alternative proof of the result in \cite{KW}, with 
applications to interpolation, in particular to the so-called Fejer problem. We also note that \cite{KW} answers Remark 5.5 in \cite{PU}.
\newline

By inspection of the proof of Corollary~\ref{corollary_KW}, we observe that 
any measure $\mu$ used in Algorithm \ref{algo_fixed_increments} yields a bound of the form
\begin{equation}
\|f-P_n^n f\|_{\mathcal L^\infty}\leq (1+2n\sqrt{K_n})\|f-P_n^\infty f\|_{\mathcal L^\infty}.
\label{equationCstLebesgue}
\end{equation}

Corollary \ref{corollary_KW} can also be stated for $m\geq n$, as done in Remark~\ref{rk_KW_for_LS}.
For instance, Algorithm~\ref{algo_fixed_increments} for $m=2n$, hence
with parameters $\delta=1/\sqrt r \leq 1/\sqrt 2$, 
$\kappa=1$, and any measure $\mu$, yields 
\[
\|f-P_n^m f\|_{\mathcal L^\infty}\leq (1+4 \sqrt{K_n})\|f-P_n^\infty f\|_{\mathcal L^\infty}.
\]
The above remarks has to be combined with upper bounds on $K_n$, for which 
we give a few examples in polynomial approximation. 

We first consider the space $V_n =\P_{n-1}$ of algebraic polynomials of degree less 
than $n$, restricted to $X=[-1,1]$. By considering the {\cnew uniform} measure
$d\mu(x) = dx/2$, we have $K_n =\sum_{j=0}^{n-1} (2j+1)= n^2$, 
while with the {\cnew arcsine} measure $d\mu(x) = \frac{dx}{\pi \sqrt{1-x^2}}$, we obtain 
$K_n =1+\sum_{j=1}^{n-1}2 =  2n-1$, {\cnew where we used uniform bounds on the Legendre and Chebyshev polynomials, respectively.}

Similar results are also available for spaces 
of multivariate polynomials indexed in lower set, see \cite[Lemma 3.1]{CCMNT}.
We recall that a lower set, or downward closed set, is a set of multi-indices 
$\Lambda \subset \N_0^d$ such that $\nu\in \Lambda$ 
implies $\nu'\in \Lambda$ for all $\nu'\leq \nu$, the last inequality being understood component-wise.
For $X=[-1,1]^d$, we consider polynomial approximation spaces of the form
\[
V_n=\Span\{x_1^{\nu_1}\dots x_d^{\nu_d},\;\nu\in \Lambda\},
\]
with $\Lambda$ a lower set of cardinality $n$. Among the associated spaces 
$V_n$, one can find polynomials of bounded degree in each variable, of bounded 
total degree, as well as hyperbolic cross spaces, which are of important use in 
high-dimensional approximation \cite{DTU}. The following estimates hold:
\begin{itemize}
\item for $\mu$ the {\cnew uniform} measure, 
$K_n = \sum_{\nu\in\Lambda} \prod_{j=1}^d (1+2\nu_j)$ 
{\cnew is 
bounded by $n^2$},
\item for $\mu$ the tensorized {\cnew arcsine} measure, 
$K_n = \sum_{\nu\in\Lambda} 2^{|\nu|_0}$ 
{\cnew is 
bounded by $n^{\frac{\log3}{\log2}}$}.
\end{itemize}
{\cnew 
The estimates $n^2$ and $n^{\frac{\log3}{\log2}}$ are established in \cite{CCMNT}.
We note also the plain linear estimates $C_\Lambda n$ where 
$C_\Lambda = \max_{\nu\in\Lambda} \prod_{j=1}^d (1+2\nu_j)$
and $C_\Lambda = \max_{\nu\in\Lambda}2^{|\nu|_0} \leq 2^d$, respectively. Such estimates 
might be much smaller, for instance in the second setting when $d$ is small}. 
Combining these bounds with \eqref{equationCstLebesgue} 
implies the results given at the end of the introduction.

In the particular case of polynomials of fixed total degree, invariance by affine transformations 
allows to treat more general cases. 

\begin{theorem}[\cite{CD_christoffel}, Theorems 5.4 and 5.6]
For $X\subset \R^d$ a compact domain, $d\mu(x)={dx}/{|X|}$ the {\cnew uniform} 
measure, and $V_n$ a set of polynomials of fixed total degree, it holds $K_n\leq C_X n^2$ if 
$X$ has a Lipschitz boundary, and $K_n \leq C_X n^{1+1/d}$ if $X$ has a  smooth boundary.
\end{theorem}

{\cnew It is worth noting that the bound $K_n \leq C_X n^2$ also holds for arbitrary lower sets when the domain is a union of rectangles of fixed volume, see Theorem 6.5 of \cite{adcock_huybrechs_2020}.}
All these results show that, in many instances, we have $K_n=O(n^2)$ with the uniform measure. Therefore Algorithm~\ref{algo_fixed_increments} provides a feasible way to sample points for which the Lebesgue constant is $\mathbb L_n=O(n^2)$.

\section{Numerical aspects}
\label{section_numerical_aspects}

Our two algorithms have a polynomial complexity in $n$, since one has to solve linear systems of size $n$ at each iteration. In fact, 
$\Tr(\bY_{i+1})$ can be computed without the knowledge of $\bY_{i+1}$, 
by taking a trace in the Shermann-Morrison formula \eqref{Sherman-Morrison},
i.e.
{\cnew
\[
\Tr(\bY_{i+1})=\Tr(\bZ_i)- \frac{ \vp(x_i)^* \bZ_i^2 \vp(x_i)}{s_i+\vp(x_i)^* \bZ_i \vp(x_i)},
\]}
so we only need a matrix inversion for $\bZ_i$. Moreover, this can be done efficiently by updating the eigenvalues and eigenvectors of $\bA_i$ from one iteration to the next, see \cite{LS_polynomial} for details.

In \cite{LS_polynomial} and \cite{LS_exponential}, polynomial and exponential modifications of the lower potential are also proposed, allowing to draw the samples by batches, and therefore reducing the complexity to {\cnew a linear expression in $n$ up to logarithmic factors, in the particular case of subsampling a graph laplacian}. The proofs are more involved, produce larger constants and require a large value of $r$, so we did not adapt them to our setting.

The essential remaining difficulty is to draw each point $x_i$ according to its prescribed density 
$\rho_i$ (or $R_i$). We start with the following observation.
\begin{remark}
\label{rk_finite_X}
{\cnew An important application is the case where $X=\{x_1,\dots,x_M\}$ is a finite set, 
and $\mu$ is the uniform measure on $X$. 
The orthonormality of the basis $\vp=(\vp_1,\dots,\vp_n)$ is encoded 
by the following decomposition of the identity into a sum of rank-one matrices: 
\[
\sum_{i=1}^M v_i v_i^*=\bI,\qquad v_i := \frac 1 {\sqrt M} (\vp_1(x_i),\dots,\vp_n(x_i))^\top.
\]
In this situation, it suffices to evaluate $\rho_i$ at all points of $X$, and to select $x_i=x$ with probability $\rho_i(x)/\sum_{x'\in X}\rho_i(x')$.
Our algorithms \ref{algo_effective_resistance} and \ref{algo_fixed_increments} can also work as a sample reduction technique, see for example \cite{CD_reduced}, where the evaluation points are extracted from a large initial sample $x_1,\dots,x_M$. Indeed, it suffices to apply our algorithms to the set $\{x_1,\dots,x_M\}$, equipped with the empirical measure $\mu=\frac{1}{M}\sum_{i=1}^M \delta_{x_i}$.}
\end{remark}

When drawing points directly from an infinite set $X$, one may use an acceptance-rejection 
strategy. Notice that $\rho_i$ in Algorithm~\ref{algo_effective_resistance} 
and $w_i$ in Algorithm~\ref{algo_fixed_increments} are of the form $\vp(x)^*\bM\vp(x)$,
with either $\bM=\bZ_i+\gamma\bI/n$ or $\bM=\bW_i$. Moreover, some routines 
have been developed \cite{CM, AC, ACD, Migliorati_irregular, Migliorati_adaptive, CD_christoffel} 
for drawing points from the Christoffel measure $\frac{1}{n}|\vp(x)|^2d\mu(x)$ in relevant 
multivariate settings.

{\cnew 
In the case of Algorithm~\ref{algo_effective_resistance}, as $\bM=\bZ_i+\gamma\bI/n$ is positive semi-definite,
it suffices to draw candidate 
points $x$ from the Christoffel measure,
and to accept them with probability   
\[
\frac{\vp(x)^*\bM\vp(x)}{\lambda_{\max}(\bM) |\vp(x)|^2} \leq 1.
\]
The probability of accepting a point is
\[
\int_{x\in X} \frac{\vp(x)^*\bM\vp(x)}{\lambda_{\max}(\bM) |\vp(x)|^2} \frac{|\vp(x)|^2}{n}d\mu(x)=\frac{\Tr(\bM)}{n\lambda_{\max}(\bM)}\geq\frac{1}{n},
\]
therefore we need in average $n\lambda_{\max}(\bM)/\Tr(\bM)\leq n$ draws from the Christoffel measure to find $x_i$.}

{\cnew 
In the case of Algorithm~\ref{algo_fixed_increments}, $\bM=\bW_i$ may have negative eigenvalues, so the bound $\lambda_{\max}(\bM)\leq \Tr(\bM)$ may not hold.
In order to draw a sample from the probability density
${R_i(x)}/{\Gamma_i}$, we can still rely on acceptance-rejection from 
the Christoffel measure, and accept points with probability   
\[
\frac{R_i(x)}{\lambda_{\max}(\bW_i) |\vp(x)|^2}
\in [0,1].
\]
The acceptance probability is then
\[
\int_{x\in X}\frac{R_i(x)}{\lambda_{\max}(\bW_i)|\vp(x)|^2}\frac{|\vp(x)|^2}{n}d\mu(x)=\frac{\Gamma_i}{n\lambda_{\max}(\bW_i)}\geq \frac{(1-\kappa)(1-\delta)^2}{n}
\]
in view of Remark \ref{lam_max_W}, and we need to at most $n(1-\kappa)^{-1}(1-\delta)^{-2}$ candidate points in average.
When $\kappa=1$, inspection of \eqref{eq_Gamma} and \eqref{lowerbound_TrW}
respectively shows that
\[
\Gamma_i\geq \Tr(\bW_i)-\frac{1-\delta}{\delta}\quad\text{and}\quad\Tr(\bW_i)\geq \frac{1-\delta}{\delta}\Tr(\bZ_i),
\]
so that
\[
\frac{\delta}{1-\delta}\Gamma_i\geq \Tr(\bZ_i)-1=\delta\Tr(\bY_i\bZ_i)\geq\delta\Tr(\bY_i^2)\geq\frac{\delta}{n}\Tr(\bY_i)^2=\frac{\delta}{n},
\]
hence we can replace the factor $1-\kappa$ in the acceptance probability by a factor $\delta/n$.

\begin{remark}
\label{remarkNumberRejection}
In the first iterations $1\leq i < n$, the maximal eigenvalue of $\bM$ has multiplicity at least $n-i$, so the bound on the acceptance probability is $n-i$ times higher. This implies for instance that there are in average less than 2 (respectively, less than $2(1-\kappa)^{-1}(1-\delta)^{-2}$)  rejections for $i=1,\dots,n/2$ in Algorithm~\ref{algo_effective_resistance} (respectively, Algorithm~\ref{algo_fixed_increments}), which we observe in practice.
\end{remark}
}

{\cnew 
Even with such methods, the computational complexity remains dominated by the sampling step: Algorithm~\ref{algo_effective_resistance} runs in $\cO(mn^3t)$ time, where the factor $m$ comes 
from the number of iterations,
a factor $n^2$ from the matrix-vector multiplication in the evaluation of $\rho_i$, the last factor $n$ 
from the average number of rejections, 
and $t$ is the time needed to generate a point with the Christoffel density. The linear algebra takes $\mathcal O(n^3)$ time per iteration, hence $\mathcal O(mn^3)$ time in total. Similarly, Algorithm~\ref{algo_fixed_increments} runs in $\cO\({mn^3t} (1-\kappa)^{-1}(1-\delta)^{-2}\)$ time, the only 
difference coming from the bound on the acceptance ratio. Therefore, applying our sampling 
strategies may prove quite challenging for large values of $n$.
}

\begin{remark}
In the deterministic setting, it is not necessary to draw each point $x_i$ with a density proportional to $\rho_i(x)$: the only requirement is that
$
w_i(x_i)\geq \kappa\frac{1-\delta}{\delta}
$
for all $1\leq i \leq m$.
The idea of using the algorithm from \cite{BSS}, still in a deterministic context, but with an infinite set $X$, seems to originate in \cite{DPSTT}, which investigates Marcinkiewicz-type discretization theorems. However, it seems that searching for $x_i$ by rejection sampling works better in practice than sorting $X$ and looking for the first point that achieves the above condition \cite{BSU}.
\end{remark}

We {\cnew next comment} on the relatively limited importance of the weights. In Algorithm~\ref{algo_effective_resistance}, the sampling density decomposes as
\[
\frac{\rho_i(x)}{\Xi_i}=(1-p) \,\frac{\vp(x)^* \bZ_i \vp(x)}{\Tr(\bZ_i)}+p\,\frac{|\vp(x)|^2}{n},
\]
with $p={\gamma}/\Xi_i$.
In others words, we are sampling
from a mixture of effective resistance and Christoffel density. Here we only know that $s_i|\vp(x_i)|^2\leq {n\eta}/{\gamma}$, {\cnew which by the way yields
\be
\label{lambdamax_Am}
\lambda_{\max}(\bA_m)\leq \Tr(\bA_m)
= \sum_{i=1}^m s_i|\vp(x_i)|^2
\leq \frac{m n\eta}{\gamma}.
\ee
But} there is no 
a priori uniform bound on the weights. If we desire such a bound, it suffices to add a third, constant term in the sampling density:
\[
\frac{\rho_i(x)}{\Xi_i}=\frac{\Tr(\bZ_i)}{\Xi_i} \frac{\vp(x)^* \bZ_i \vp(x)}{\Tr(\bZ_i)}+\frac{\gamma}{\Xi_i} \frac{|\vp(x)|^2}{n}+\frac{\gamma_\infty}{\Xi_i}
\]
for a new parameter ${\cnew \gamma_\infty > 0}$, where we redefine $\Xi_i$ as $\Tr(\bZ_i)+\gamma+\gamma_\infty$. To choose a point according to this density, one can simply draw it from 
$d\mu(x)$ with probability $\gamma_\infty/\Xi_i$, from the Christoffel measure with probability
$\gamma/\Xi_i$, and from the effective resistance measure otherwise. This choice immediately 
yields $s_i=\eta/\rho_i(x_i)\leq \eta/\gamma_\infty$, however taking $\gamma_\infty>0$ deteriorates a 
bit the other estimates, which is why we did not include it earlier.

{\cnew Having a uniform bound on the weights $s_i$ is particularly important in case of noisy samples, where one can only observe $f(x_i)+\eps_i$ for some random $\eps_i$. Then the estimates of Theorem~\ref{main_theorem} hold with an additional noise term of the form
\[
\|\bL^+(\sqrt{s_i}\eps_i)_{1\leq i \leq m}\|_2^2\leq \lambda_{\min}(\bA_m)^{-1}\frac{\eta}{\gamma_\infty}\|(\eps_i)_{1\leq i \leq m}\|_2^2,
\]
due the definition of the weighted least-squares estimator in Section~\ref{SectionLS}. We refer to Theorems~5.3 and 5.19 of \cite{ABW} for details.
}
\newline

Concerning Algorithm~\ref{algo_fixed_increments}, in the deterministic setting, one can replace each weight $s_i$ by its upper bound $\frac{\delta}{\kappa(1-\delta)}$, resulting in an unweighted discrete norm
\[
\|g\|_m^2=\frac{1}{m}\sum_{i=1}^m |g(x_i)|^2,
\]
without changing our estimates. We refer to \cite{BSU} for earlier results on subsampling of frames with unweighted discrete norms.
In any case, in numerical experiments, the weights seem to have lesser importance than the position of the points.
\newline

{\cnew
\section {Numerical experiments}
\label{sec:numerical_illustration}
We conclude with a few illustrations in the space of multivariate polynomials, focusing on a function whose best approximation is explicitly known.
Our code is available at the following address:
\href{https://github.com/Belloliva/minimal-oversampling-for-multivariate-polynomial-least-squares}{\small github.com/Belloliva/minimal-oversampling-for-multivariate-polynomial-least-squares}.\newline

\noindent In the univariate setting, the Legendre polynomials $(L_k)_{k\geq 0}$ are orthonormal in $L^2([-1,1],dx/2)$,
and satisfy
\[
\|L_k\|_{L^\infty}=L_k(1)=\sqrt{2k+1}.
\]
We recall that their generating 
function is given by 
\[
g_y(x) =\sum_{k=0}^\infty \frac{L_k(x)}{L_k (1)}\, y^k = \frac {1}{\sqrt{1-2xy+y^2}},\quad x\in[-1,1],\quad y\in (-1,1). 
\]
Now, on the domain $X=[-1,1]^d$ with $d\in \N$, one can consider tensorized Legendre polynomials
\[
L_{\bk}(\bx) := \prod_{j=1}^d L_{k_j}(x_j),\qquad \bx=(x_1,\dots,x_d)\in X,\quad \bk=(k_1,\dots,k_d)\in\N_0^d,
\]
which are orthonormal for the uniform measure $d\mu(\bx)=d\bx/2^d$ over $X$.

Fixing $\by =(y_1,\dots,y_d) \in (0,1)^d$, we would like to approximate
the multivariate function
\[
\bx \mapsto g_\by (\bx) =\prod_{j=1}^d g_{y_j}(x_j)
=\sum_{\bk\in \N_0^d} c_\bk L_{\bk}(\bx), \qquad c_\bk = \prod_{j=1}^d \frac{y_j^{k_j}}{\sqrt{2k_j+1}}.
\]
This is a typical representer of holomorphic functions on $X$ with 
anisotropic dependance on $\bx$ pa\-ra\-me\-te\-rized by $\by$.
Its $L^2(X,\mu)$ norm is analytically given by
\[
\|g_\by\|_{L^2}^2=\prod_{j=1}^d \int_{-1}^1 |g_{y_j}(x_j)|^2\frac{dx_j}{2}=\prod_{j=1}^d \frac 1{2y_j} \ln\(\frac{1+y_j}{1-y_j}\).
\]
Moreover, the best approximation space of dimension $n$ for $g_\by$ w.r.t. $\|\cdot\|_{L^2}$ is
\[
V_n=\Span\{L_\bk,\;\bk \in \Lambda_n\},
\]
where $\Lambda_n$ is the set of $n$ multi-indices $\bk$ with largest values of $c_\bk$.
Note that $\Lambda_n$ is lower and is easily computed, since the coefficients $c_\bk$ 
are decreasing for the component-wise order,
and that the approximation error is
\be
\min_{v\in V_n}\|g_\by-v\|_{L^2}^2=\Big\|g_\by-\sum_{\bk\in \Lambda_n}c_\bk L_\bk\Big\|_{L^2}^2=\|g_\by\|_{L^2}^2-\sum_{\bk\in \Lambda_n}c_\bk^2.
\label{optimum_Legendre}
\ee
\begin{remark}
When looking at approximations in the uniform norm, although $V_n$ and $\displaystyle{\sum_{\bk\in\Lambda_n}c_\bk L_\bk}\vspace{-3mm}$ may not be optimal, one can still compute an upper bound
\[
\min_{v\in V_n}\|g_\by-v\|_{L^\infty}\leq\Big\|g_\by-\sum_{\bk\in \Lambda_n}c_\bk L_\bk\Big\|_{L^\infty}=\prod_{j=1}^d \frac{1}{1-y_j}-\sum_{k\in \Lambda_n} \prod_{j=1}^d y_j^{k_j},
\]
since the maximum is attained at $(1,\dots,1)\in X$.
\end{remark}

In Figure~\ref{fig:condition_number}, we take parameters $\by=(0.9,0.8,0.7,0.6)$ in spatial dimension $d=4$, use Legendre expansions of size $n=128$,
and draw $m=2n=256$ points ${\bx_1,\dots,\bx_m\in X}$ according to one of the following strategies:
\begin{enumerate}
\item[(a)] The points are i.i.d according to the uniform measure $d\mu(\bx)$
\item[(b)] The points are i.i.d according to the tensor product arcsine measure\vspace{-2mm}
\[
\prod_{j=1}^d \frac{dx_j}{\pi \sqrt{1-x_j^2}}\vspace{-2mm}
\]
\item[(c)] The points are i.i.d according to the Christoffel measure\vspace{-2mm}
\[
\frac 1n\sum_{\bk\in\Lambda_n} |L_\bk(\by)|^2\,d\mu(\bx)\vspace{-2mm}
\]
\item[(d)] The points and weights are generated by Algorithm~\ref{algo_effective_resistance} with input parameters $\eps = r^{-1/4}$ and $\gamma=r^{1/2}-r^{1/4}$, where $r=(m+1)/n$
\item[(e)] The points and weights are generated by Algorithm~\ref{algo_fixed_increments} with input parameters $\delta = r^{-1/2}$ and $\kappa =1/2$, where $r=m/(n-1)$.
\end{enumerate}
\begin{figure}[ht]
 \begin{subfigure}{0.19\textwidth}
     \includegraphics[width=\textwidth]{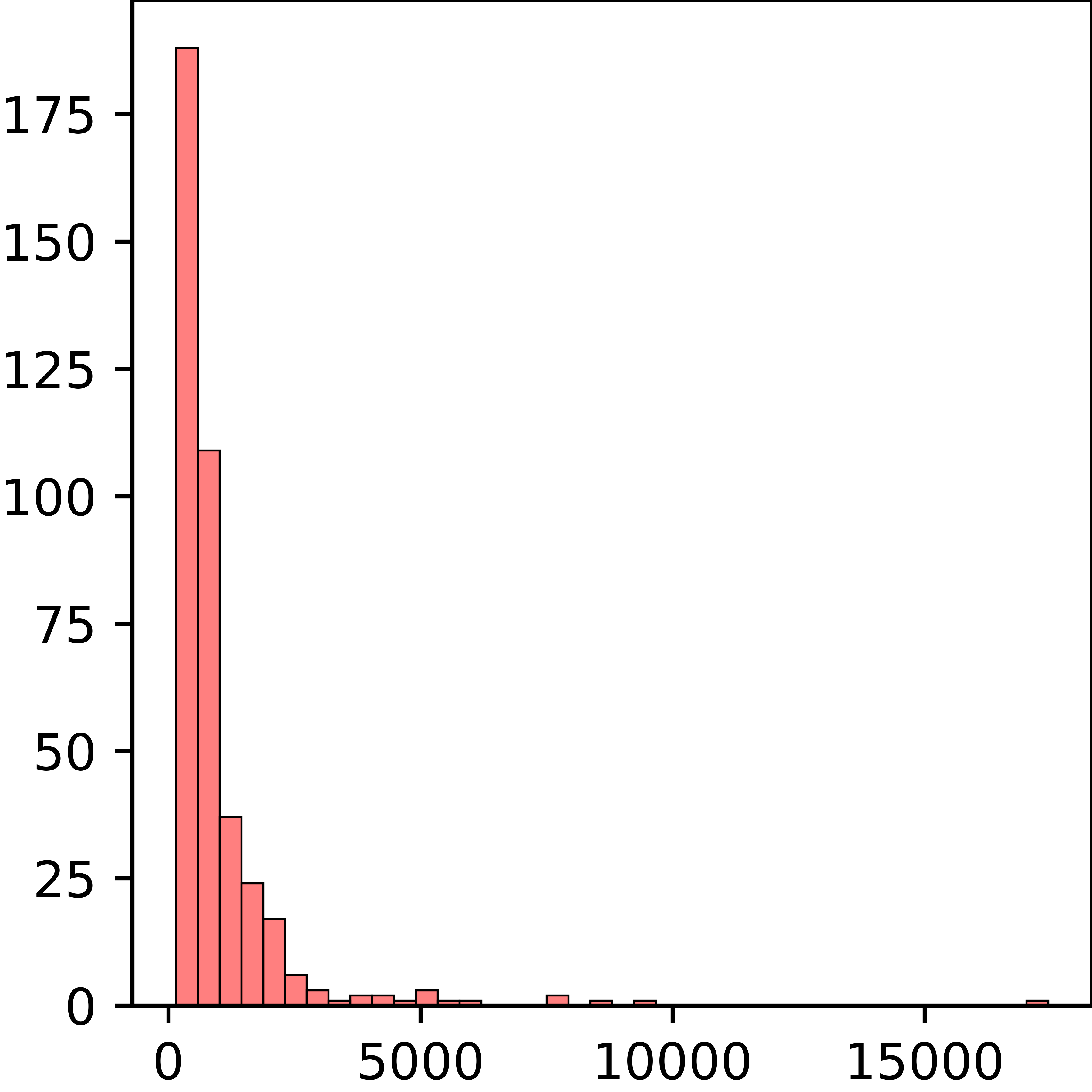}
     \caption{Uniform}
     \label{fig:uniform}
 \end{subfigure}
 \begin{subfigure}{0.19\textwidth}
     \includegraphics[width=\textwidth]{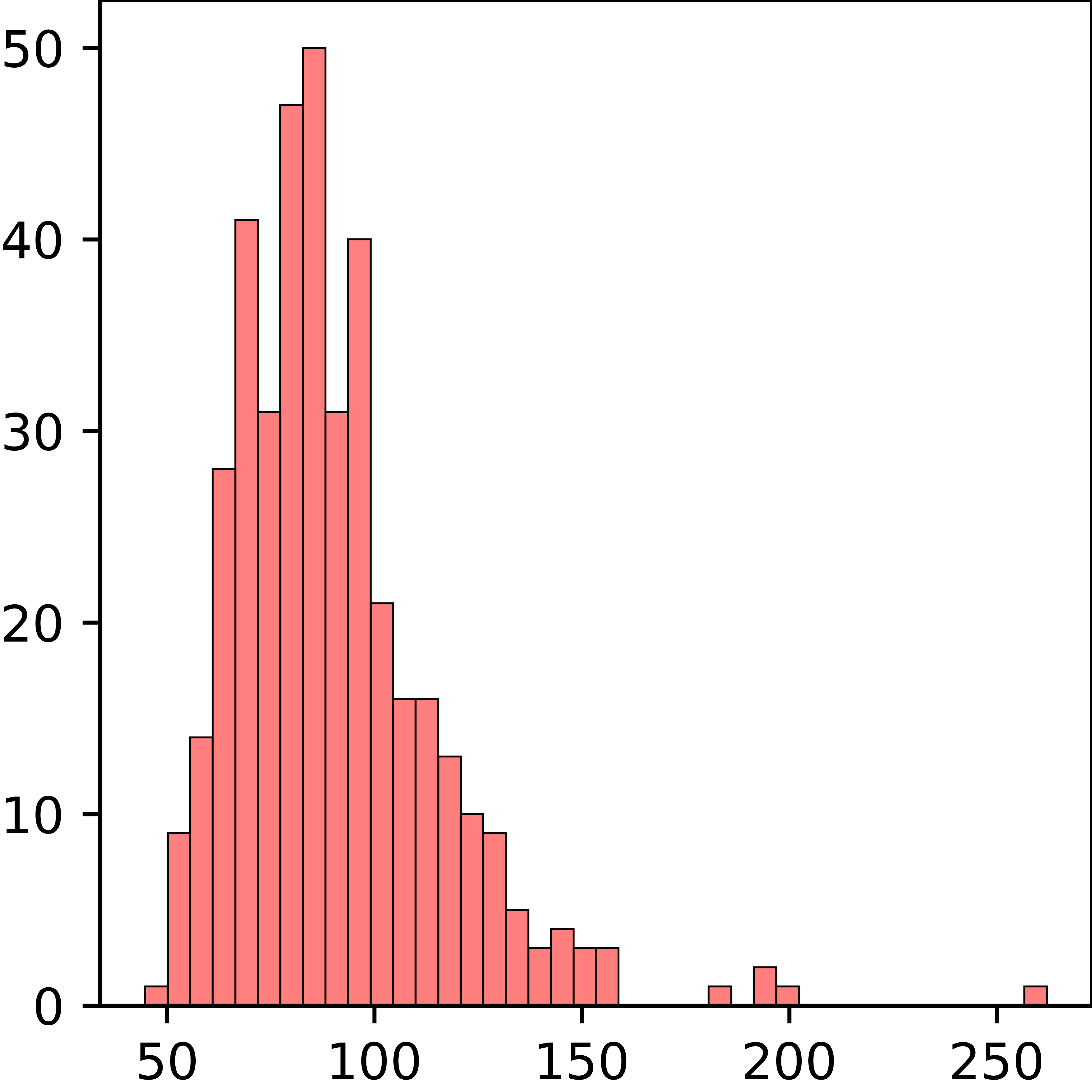}
     \caption{Arcsine}
     \label{fig:arcsine}
 \end{subfigure}
 \begin{subfigure}{0.19\textwidth}
     \includegraphics[width=\textwidth]{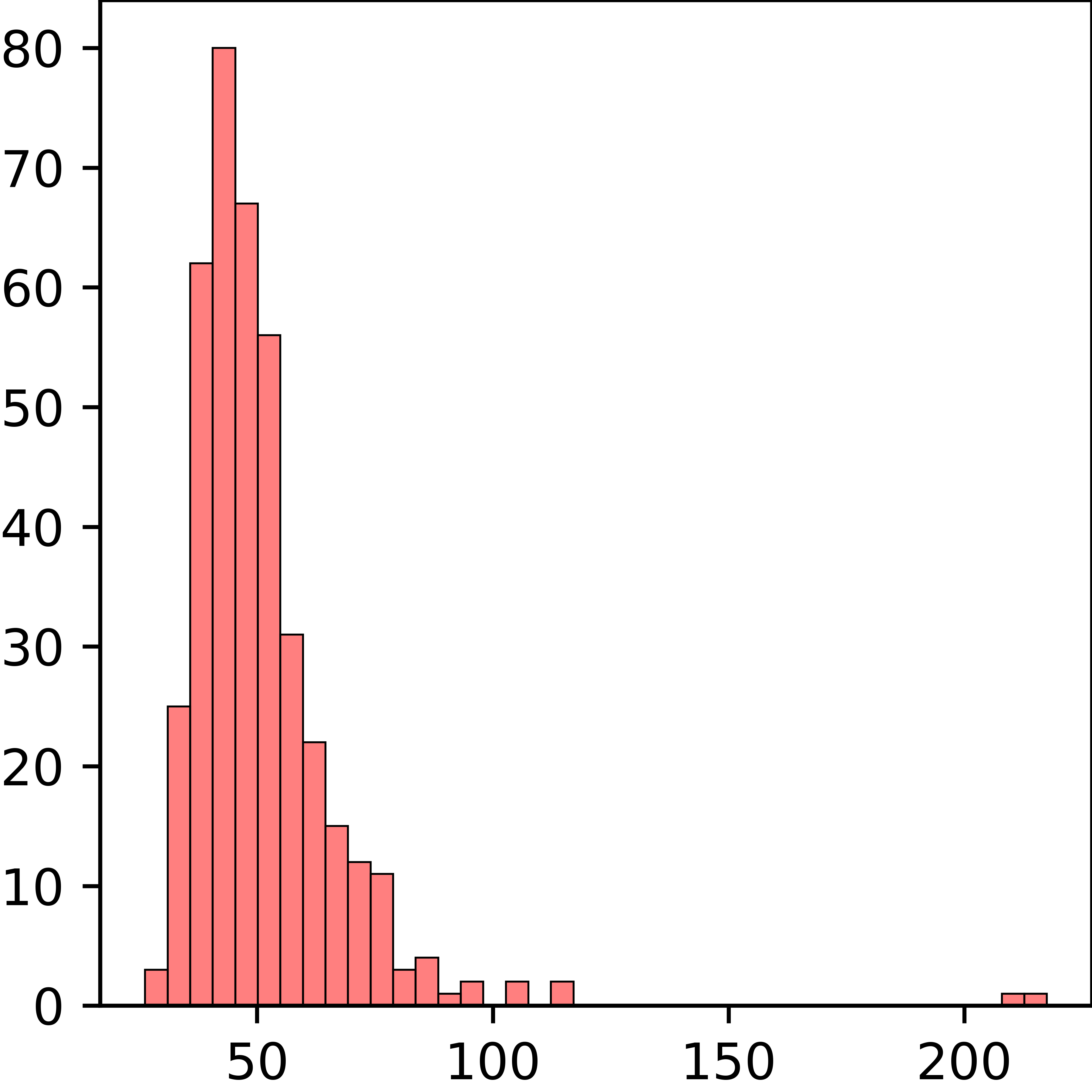}
     \caption{Christoffel}
     \label{fig:christoffel}
 \end{subfigure}
 \begin{subfigure}{0.19\textwidth}
     \includegraphics[width=\textwidth]{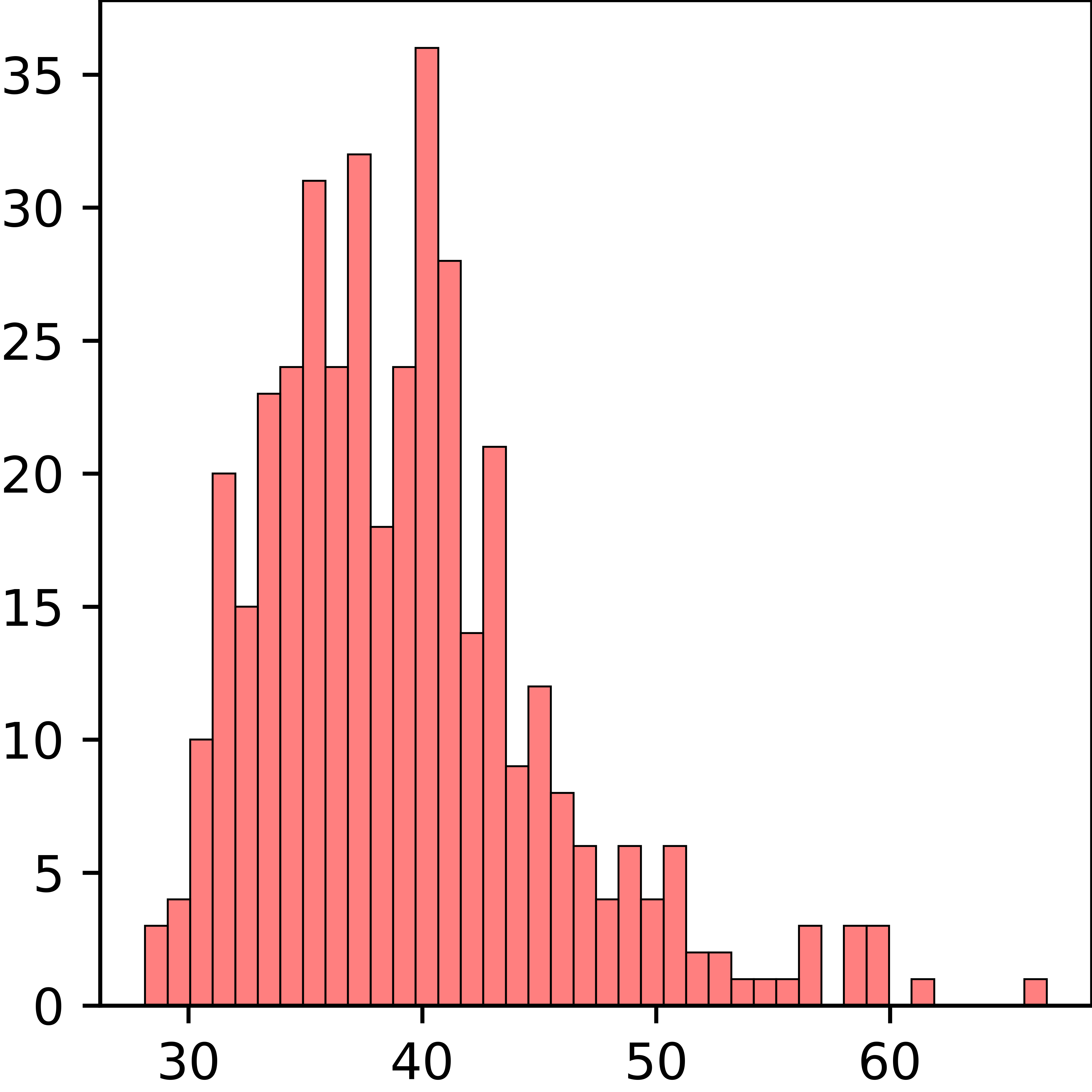}
     \caption{Algorithm~\ref{algo_effective_resistance}}
     \label{fig:effective_resistance}
 \end{subfigure}
 \begin{subfigure}{0.19\textwidth}
     \includegraphics[width=\textwidth]{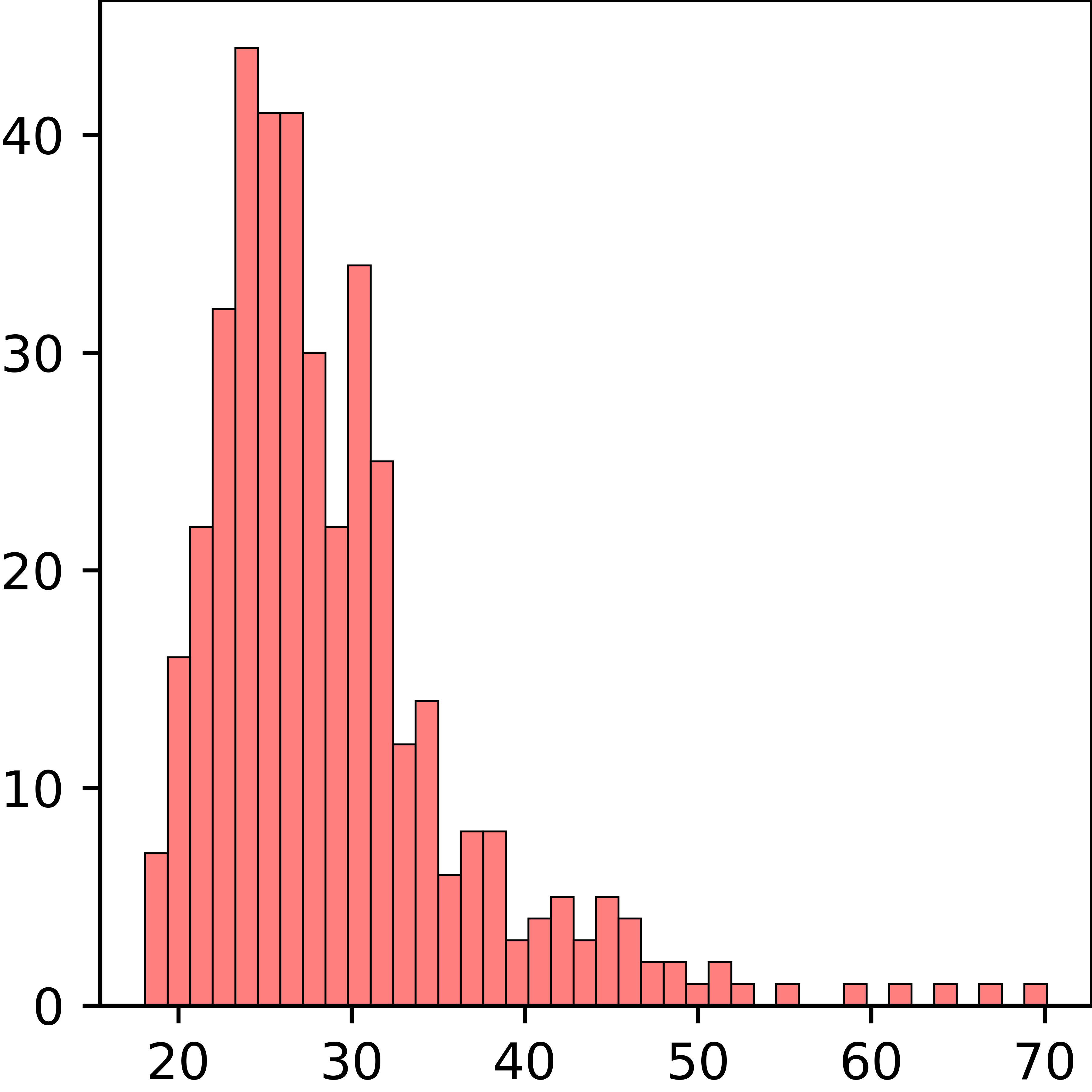}
     \caption{Algorithm~\ref{algo_fixed_increments}}
     \label{fig:fixed_increments}
 \end{subfigure}
\caption{Histograms of condition number of $\bA_m$ for various random sampling strategies}
\label{fig:condition_number}
\end{figure}

In the first three cases, the weights $s_i$ are taken as the inverse of the sampling density 
(w.r.t. $d\mu(\bx)$) at points $\bx_i$. The uniform (a) and Christoffel (c) settings are the ones studied in \cite{CDL} and \cite{CM}, and the arcsine measure (b) is the limit of the Christoffel measure when $d$ is fixed and $n$ tends to infinity, thus yielding similar properties while being slightly simpler to sample from.

Each scheme outputs a random matrix $\bA_m=(\sum_{i=1}^m s_i L_\bk(\bx_i)L_{\bk'}(\bx_i))_{\bk,\bk'\in \Lambda_n}$, and
Figure~\ref{fig:condition_number} displays histograms of their condition number $\lambda_{\max}(\bA_m)/\lambda_{\min}(\bA_m)$ over 400 runs.
Indeed, this condition number is a classical indicator of the accuracy and robustness of least-squares, compared to the best possible approximation.

We immediately observe that the last two methods achieve smaller condition numbers, often between 30 and 40 for Algorithm~\ref{algo_effective_resistance}, and between 20 and 30 for Algorithm~\ref{algo_fixed_increments}, compared to i.i.d points, which generally give values larger than 40, and sometimes much larger, especially in the case of uniform points.
Heuristically, we expect the condition number of Algorithm~\ref{algo_fixed_increments} to be in average a bit larger than the eigenvalues of
\[
\frac{\E(\bA_m)}{\lambda_{\min}(\bA_m)}\preccurlyeq \frac{(1-\kappa)^{-1}m\delta(1-\delta)^{-1}}{(n-1)(\delta r-1)}\bI=\frac{2}{(1-1/\sqrt r)^2}\bI \preccurlyeq 22.9\, \bI,
\]
where we used arguments from the proof of Theorem \ref{main_theorem}, equation \eqref{main_equation_r_small}, together with the fact that $r=\frac{256}{127}$.
Therefore, there is a very good agreement between the theoretical and numerical spectral properties of $\bA_m$.

It should be mentioned that in Algorithm~\ref{algo_effective_resistance}, the random event from Proposition~\ref{prop_Am_algo1} was always realized, resulting in usual least-squares instead of the conditioned version \eqref{conditional_WLS}. Nevertheless, redrawing the whole sample when the final condition number is large, as proposed in \cite{HNP} for Christoffel points, remains a good option in practice for all sampling schemes, since it amounts to truncating the tails of the above histograms.

In Figure \ref{fig:rejections_number}, we plot the number of rejections as a function of the 
iteration index $i=1,\dots,m$, averaged over the 400 runs, for Algorithms \ref{algo_effective_resistance} and \ref{algo_fixed_increments}.
In accordance with Remark \ref{remarkNumberRejection}, few 
rejections are observed in the first iterations $i< n$. However, the number of rejections remains moderate in the second half $n\leq i \leq m$, staying below 15 in the case of Algorithm~\ref{algo_effective_resistance}, which is much better than the pessimistic upper bound $n=128$ from Section~\ref{section_numerical_aspects}.
Algorithm \ref{algo_fixed_increments} incurs about 7 times more rejections, which is again smaller than the factor
$(1-\kappa)^{-1}(1-\delta)^{-2}\simeq 23$ encountered in the analysis.

\begin{figure}[ht]
\centering
 \begin{subfigure}{0.48\textwidth}
     \includegraphics[width=\textwidth]{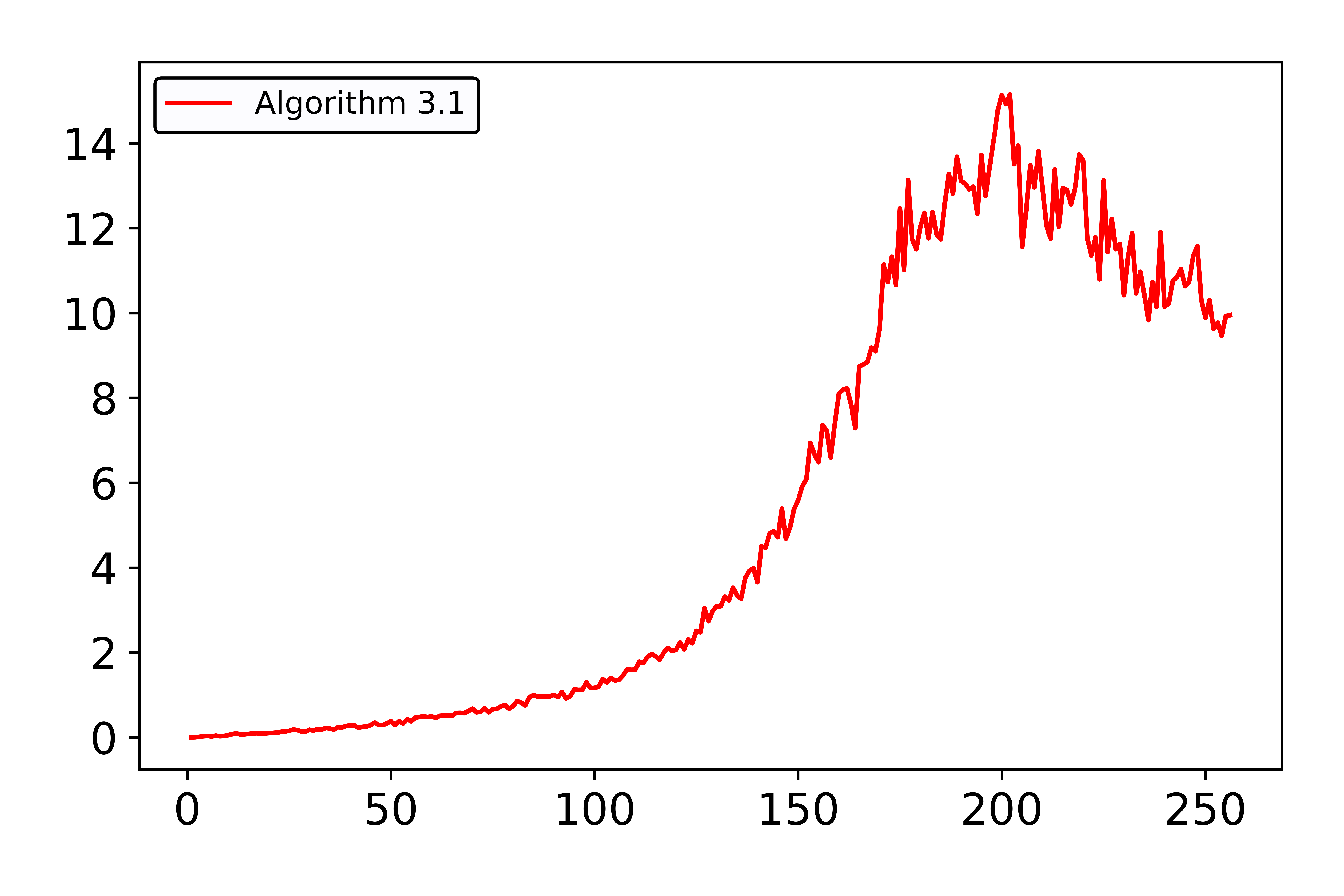}
 \end{subfigure}
 \begin{subfigure}{0.48\textwidth}
     \includegraphics[width=\textwidth]{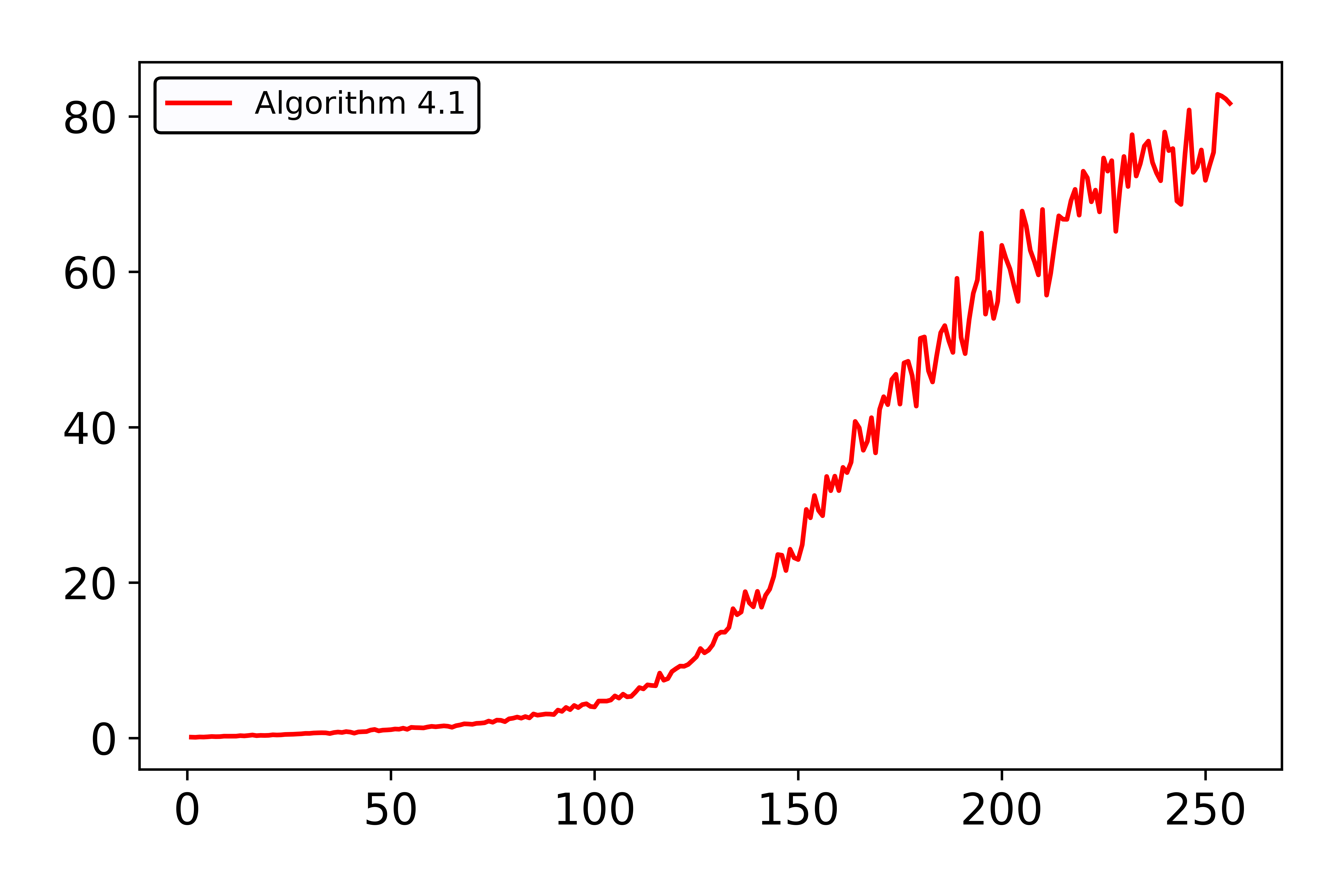}
 \end{subfigure}
\caption{Average over 400 runs of the number of rejections in each iteration}
\label{fig:rejections_number}
\end{figure}

Finally, in Figure~\ref{fig:error_plot}, we keep the same parameters $d=4$, $\by=(0.9,0.8,0.7,0.6)$ and $n=128$, and 
plot the normalized error 
\[
\E_{\rm emp} (\|g_\by - P_n^m g_\by \|_{L^2})/ E^*
\]
as a function of the number of samples $m$, for $m=n+2p$ and $p=0,\dots,40$. Here 
$\E_{\rm emp}$ stands for an empirical average over 100 runs, and 
$E^*\simeq 0.402882$ is the error of best approximation, computed 
by \eqref{optimum_Legendre}. For the sake of clarity, we only display the results for the 
Christoffel measure and our algorithms. We again observe that the latter perform better, 
especially in the regime $m\approx n$, and that they are within a factor 2 of the optimal error 
as soon as $m\geq n+30$.

\begin{figure}[ht]
\centering
\includegraphics[width=0.5\textwidth]{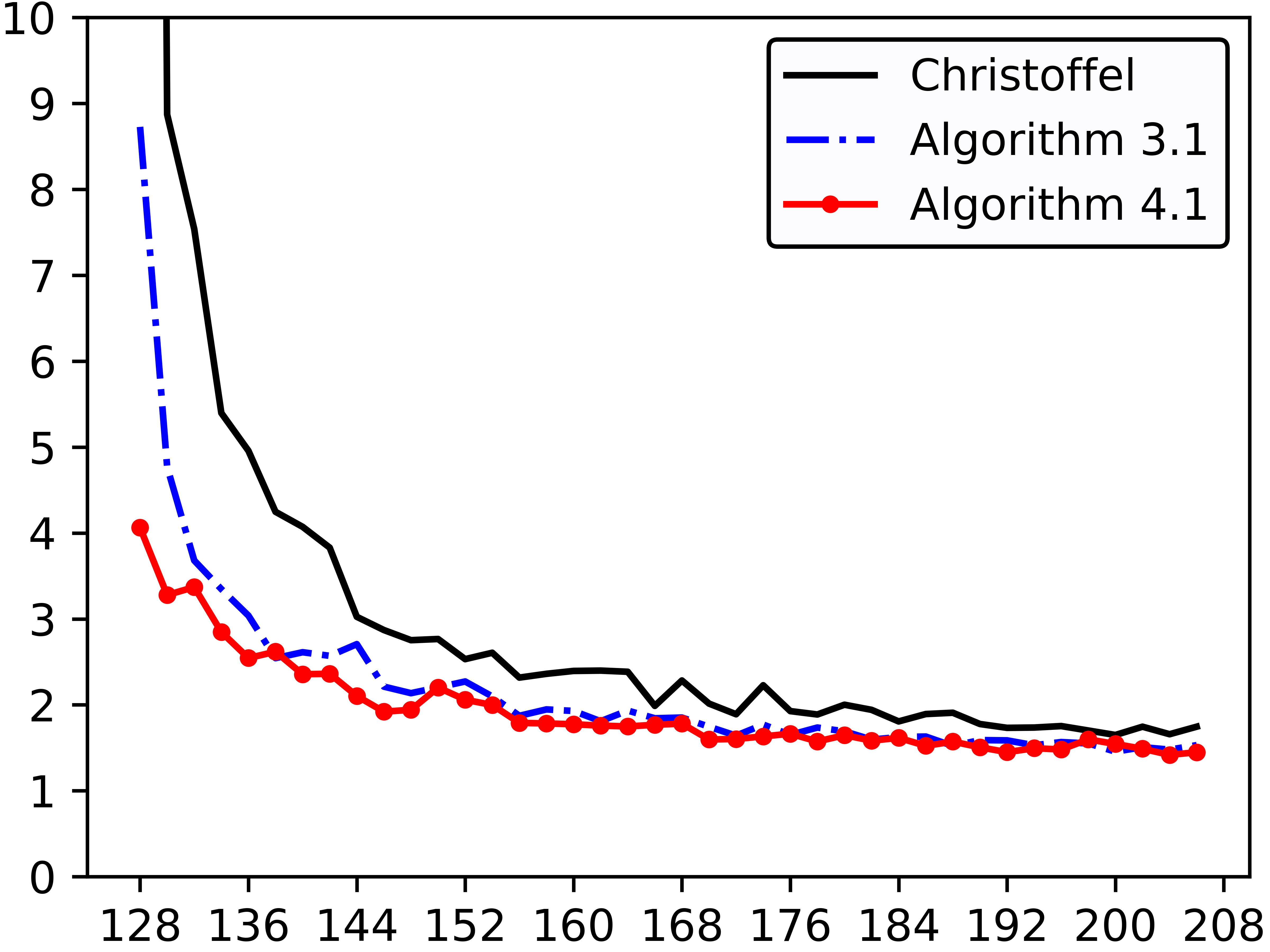}
\caption{Ratio between the least-squares error and the best approximation error for $n=128$ and $m$ ranging between 128 and 208, when the points and weights are i.i.d according to the Christoffel measure, or generated by Algorithms~\ref{algo_effective_resistance} and \ref{algo_fixed_increments}.}
\label{fig:error_plot}
\end{figure}

\begin{remark}
In view of the discussion from Section~\ref{section_numerical_aspects}, the last three schemes necessitate sampling from 
Christoffel measure, which can be efficiently implemented by
first drawing $\bk$ uniformly in $\Lambda_n$, and then drawing $\bx$ from the measure
\[
|L_\bk(\bx)|^2 d\mu (\by) 
= \prod_{j=1}^d |L_{k_j}(x_j)|^2 \frac{dx_j}{2}.
\]
For the second step, exploiting the product structure of $L_\bk$, it suffices to draw each component $x_j$ independently from
$|L_{k_j}(x)|^2 \frac{dx}{2}$.
When $k_j=0$, this is just the uniform measure over $[-1,1]$.
Otherwise, we rely once more on acceptance/rejection from the arcsine 
measure, that is from the cosine of uniform points. Thanks to the so-called Berstein inequality for Legendre 
polynomials   
\[
\frac {|L_k(x)|^2}{2} \leq  \frac {2k+1}{k} \frac {1}{\pi \sqrt{1-x^2}},
\]
the acceptance probability is at least $k_j/(2k_j+1)\geq 1/3$ for $k_j\geq 1$, so there are at most $2\supp(\bk)\leq 2\min(n,d)$ rejections in average. As a conclusion, the time needed to generate a point from the Christoffel probability measure is $t=\mathcal O(d\min(n,d))$.
We refer to Section 5 of \cite{CM}, as well as \cite{CD_christoffel} and \cite{Migliorati_irregular}, for an overview of efficient sampling strategies on more general domains.
\end{remark}
}

\noindent{\bf Aknowledgement:} The authors would like to thank Albert Cohen for insightful feedback
and discussions all along the elaboration of the paper, David Krieg, Mario Ullrich and Tino Ullrich for their enriching questions and comments, {\cnew and the reviewers for their careful reading and valuable suggestions.}


\bibliographystyle{alpha}
\bibliography{ChkifaDolbeault_minimal_oversampling}
\end{document}